\RequirePackage{fix-cm}
\documentclass[smallextended]{svjour3}       % onecolumn (second format)
\smartqed  % flush right qed marks, e.g. at end of proof
\usepackage{csquotes}
\usepackage{graphicx}
\usepackage{amsfonts}
\usepackage{algorithm}
\usepackage{algorithmic}
\usepackage{amsmath}
\usepackage{geometry}
\usepackage{color}
\usepackage[hidelinks]{hyperref}
\usepackage{cite}

\begin{document}

\title{A novel interpolation–regression approach for function approximation on the disk and its application to cubature formulas
}
% \subtitle{Do you have a subtitle?\\ If so, write it here}

\titlerunning{Function approximation on the disk}        % if too long for running head

\author{Francesco Dell'Accio \and 
        Francisco Marcellán \and 
        Federico Nudo 
}


\institute{  Francesco Dell'Accio \at
             Department of Mathematics and Computer Science, University of Calabria, Rende (CS), Italy\\ 
             \email{francesco.dellaccio@unical.it} 
         \and 
             Francisco Marcellán \at
              Departamento de Matemáticas, Universidad Carlos III de Madrid, Spain \\
              \email{pacomarc@ing.uc3m.es}
         \and 
            Federico Nudo (corresponding author) \at
              Department of Mathematics and Computer Science, University of Calabria, Rende (CS), Italy \\
               Istituto Nazionale di Alta Matematica, Roma, Italy\\
              \email{federico.nudo@unical.it}
}

\date{Received: date / Accepted: date}

\maketitle

\begin{abstract}

The interpolation–regression approximation is a powerful tool in numerical analysis for reconstructing functions defined on square or triangular domains from their evaluations at a regular set of nodes. The importance of this technique lies in its ability to avoid the Runge phenomenon. In this paper, we present a polynomial approximation method based on an interpolation–regression approach for reconstructing functions defined on disk domains from their evaluations at a general set of sampling points. Special attention is devoted to the selection of interpolation nodes to ensure numerical stability, particularly in the context of Zernike polynomials. As an application, the proposed method is used to derive accurate cubature formulas for numerical integration over the disk.

\end{abstract}

\section{Introduction}
The approximation of functions on disk domains has long been a topic of significant interest in applied mathematics~\cite{wermer1964approximation, trefethen1981rational, gutknecht1983nonuniqueness, hangelbroek2008error}. In many practical applications, such as optical engineering~\cite{born1970principles}, aberrometry of the human eye~\cite{carvalho2005accuracy}, corneal surface modeling~\cite{smolek2003zernike, smolek2005goodness}, and other topics in optics and ophthalmology~\cite{nam2009zernike}, the domain of interest is naturally circular. Unisolvent configurations of nodes for polynomial approximation over these domains have been studied in~\cite{bojanov2003polynomial}. However, in real-world applications, we often face the challenge of reconstructing a function defined on a disk domain from its values at a discrete set of points, typically obtained by sampling a physical phenomenon. Although standard polynomial interpolation is conceptually straightforward, it suffers from numerical instabilities, especially when the interpolation nodes are not carefully chosen.

\noindent
The problem of reconstructing a function using a polynomial approximation that fully exploits the information available from a regular set of nodes has been previously addressed through the introduction of the constrained mock-Chebyshev least squares operator for square domains~\cite{DellAccio:2022:GOT, DellAccio:2024:AEO}, and later extended to triangular domains~\cite{de2024mixed}. The core idea behind this approach is to interpolate the unknown function on a carefully selected subset of nodes, chosen to be close to an optimal configuration for the interpolation process—such as Chebyshev nodes on the interval, Padua nodes on the square, and Leja nodes on the triangle. The remaining nodes are then employed for a simultaneous regression to improve the accuracy of the approximation. This strategy has been successfully applied to develop quadrature formulas~\cite{DellAccio:2022:CMC, DellAccio:2022:AAA, DellAccio:2023:PIR, dell2024quadrature}, to numerically solve Fredholm integral equations~\cite{DellAccio:2024:NAO}, to approximate derivatives~\cite{DellAccio:2024:PAO}, and has recently been extended to Hermite interpolation~\cite{dell2025constrained}.

\medskip

The goal of this work is twofold: \begin{itemize}
\item First, we extend the constrained mock-Chebyshev least squares framework to functions defined on a disk domain, thereby broadening its applicability to circular geometries.
\item Second, we apply the proposed function approximation technique to define accurate cubature formulas for numerical integration on the disk. Accurate integration over disk domains is crucial in many fields, including image analysis and physics~\cite{boersma1993solution, shevchuk2003impingement, punzi2008meshless, ramos2016optimal, ramaciotti2017some, sommariva2017numerical, sommariva2022cubature}.
\end{itemize}

\noindent 
The paper is organized as follows. In Section~\ref{SecInterp}, we recall the definition of Zernike polynomials and the Bos array. We then introduce our novel interpolation–regression operator for disk domains and provide a complete analysis of the behavior of its operator norm. In Section~\ref{SecCubFor}, we apply the proposed approximation framework to the construction of cubature formulas on the disk. Finally, in Section~\ref{SecNumEx}, we present several numerical experiments that illustrate the effectiveness of our approach.

\section{Polynomial approximation on the disk via an interpolation-regression approach}\label{SecInterp}
Let $r > 0$ and consider the disk of radius $r$ 
\begin{equation*}
    D_r = \left\{ \left(r \cos(\theta), r \sin(\theta)\right) \, : \, \theta \in [0,2\pi) \right\}.
\end{equation*}
Without loss of generality, we assume that $r=1$. Our goal is to reconstruct an unknown real-valued function $f$ defined on $D_1$, using only its evaluations on a set of sampling points. Specifically, we consider the set
\begin{equation}\label{pointset}
    S_{n,k}=\left\{\boldsymbol{x}_{\eta ,\kappa}=\left(r_\eta\cos \left(\theta_\kappa\right),r_\eta\sin\left(\theta_\kappa\right)\right)\, :\, \eta=0,\dots,n, \, \kappa=0,\dots,k\right\}\subset D_1,
\end{equation}
where $0<r_0 < r_1 < \cdots < r_n \leq 1$ and $0 \leq \theta_0 < \theta_1 < \cdots < \theta_k < 2\pi$, are increasing sequences of radii and angles, respectively.  We notice that
\begin{equation*}
    \#\left(S_{n,k}\right) = (n+1)(k+1),
\end{equation*}
where $\#(\cdot)$ denotes the cardinality operator. In general, the number of these points does not match the dimension of a bivariate polynomial space, and thus a unique polynomial interpolant might not exist. This motivates the development of an interpolation-regression method that uses all the points of $S_{n,k}$ to construct a high quality polynomial approximation.

\subsection{Optimal set of interpolation nodes on the disk}

The Zernike polynomials are usually defined in polar coordinates using the double-index notation~\cite{von1934beugungstheorie, dunkl2014orthogonal, ramos2016optimal} as
\begin{equation*}
Z_m^l(r, \theta) := \gamma_m^lR_m^{|l|}(r)  \chi_l(\theta), \quad \text{with} \quad m\in\mathbb{N}, \quad  l\in \mathbb{Z},\quad \left|l\right| \leq m,\quad m - \left|l\right| \in 2\mathbb{N}, 
\end{equation*}
where the radial part is given by
\begin{equation*}
R_m^{|l|}(r) := \sum_{s=0}^{\frac{m - |l|}{2}} (-1)^s \frac{(m - s)!}{s! \left(\frac{m + |l|}{2} - s\right)! \left(\frac{m - |l|}{2} - s\right)!} r^{m - 2s},
\end{equation*}
the angular part is given by
\begin{equation*}
\chi_l(\theta) := 
\begin{cases}
\cos(l\theta), & \text{if } l \geq 0, \\
\sin(|l|\theta), & \text{if } l < 0,
\end{cases}    
\end{equation*}
and $\gamma_m^l$ are normalization constants~\cite{ramos2016optimal}. Equivalently, the radial component $R_m^{|l|}(r)$,  
can be written as
 \begin{eqnarray*}
R_m^{|l|}(r)
&=& (-1)^{\frac{m - |l|}{2}}\, r^{|l|} P_{\frac{m - |l|}{2}}^{(|l|,0)}\left(1 - 2r^2\right) \\
&=& \binom{m}{\frac{m + |l|}{2}} r^m
   {}_2F_{1}\left(-\frac{m + |l|}{2},-\frac{m - |l|}{2}; -m; r^{-2}\right) \\
&=& (-1)^{\frac{m - |l|}{2}} \binom{\tfrac{m + |l|}{2}}{|l|}\, r^{|l|}
   {}_2F_{1}\left(1 + \frac{m + |l|}{2}, -\frac{m - |l|}{2}; 1 + |l|; r^2\right),
\end{eqnarray*}
where $P_{n}^{(\alpha,\beta)}$ is the classical Jacobi polynomial of degree $n$ with parameters $(\alpha,\beta)$ and ${}_2F_{1}$ is the Gaussian hypergeometric function.\\

\noindent
For a fixed $m\in\mathbb{N}$, the Zernike polynomials satisfy  
\begin{equation*}
\operatorname{span}\left\{Z_m^l(r,\theta) \, :\, l\in\mathbb{Z}, \, |l|\le m, \, m-|l|\in 2\mathbb{N}\right\}=\mathbb{P}_m\left(\mathbb{R}^2\right),
\end{equation*}
where $\mathbb{P}_m\left(\mathbb{R}^2\right)$ denotes the space of bivariate polynomials of total degree $m$. Moreover, they constitute an orthonormal basis of $\mathbb{P}_m\left(\mathbb{R}^2\right)$ on the unit disk with respect to the weight $rd\theta dr$, that is
\begin{equation*}
\left\langle Z_m^l, Z_{m^\prime}^{l^\prime} \right\rangle_{D_1} := \int_0^1 \int_0^{2\pi} Z_m^l(r, \theta) Z_{m'}^{l^\prime}(r, \theta)\, r  d\theta dr = 0, \quad  (m, l) \neq \left(m^\prime, l^\prime\right),
\end{equation*}
and
\begin{equation*}
\left\langle Z_m^l, Z_{m}^{l} \right\rangle_{D_1} = 1.
\end{equation*}

\noindent
To ensure numerical stability when interpolating functions using the Zernike polynomial basis, the selection of interpolation nodes on the disk is crucial. A common metric for assessing a node configuration is the condition number of the collocation matrix obtained by evaluating the Zernike polynomial basis at the nodes. Among the various known unisolvent configurations on the disk, one of the most effective is the \emph{Bos array}~\cite{BosThesis, bojanov2003polynomial, ramos2016optimal}, which arranges nodes along concentric circles determined by a given maximal radial order $m$. More precisely, let us define a sequence of decreasing radii
\begin{equation*}
1 \ge \rho_1 > \rho_2 > \cdots > \rho_K \ge 0, \quad K = K(m) = \left\lfloor \frac{m}{2} \right\rfloor + 1.    
\end{equation*}
On each circle of radius $\rho_\nu$, we place
\begin{equation*}
    m_\nu=2m+5-4\nu
\end{equation*}
equispaced nodes. In other words, on each circle of radius $\rho_\nu$, the nodes are defined by 
\begin{equation*}
\boldsymbol{p}_{\nu,\sigma} = \left( \rho_\nu\cos\left(\frac{2\pi \sigma}{m_\nu}\right),  \rho_\nu\sin\left(\frac{2\pi \sigma}{m_\nu}\right) \right), \quad \sigma = 0,1,\dots, m_\nu-1.    
\end{equation*}
It is straightforward to verify that
\begin{equation*}
    \sum_{\nu=1}^K m_\nu=\frac{(m+1)(m+2)}{2}=:M,
\end{equation*}
which exactly matches the dimension of $\mathbb{P}_m\left(\mathbb{R}^2\right)$. \\

\noindent
To simplify the notation, we sometimes re-index the nodes and the Zernike polynomials using a single index, that is
\begin{eqnarray}
    \boldsymbol{x}_{\eta,\kappa}&\rightarrow& \boldsymbol{x}_{i}, \quad i=1,\dots,(n+1)(k+1), \label{not1}\\
     \boldsymbol{p}_{\nu,\sigma}&\rightarrow& \boldsymbol{p}_{j}, \quad j=1,\dots,M,  \label{not2}\\
      Z_{m}^{l}&\rightarrow& Z_{\ell}, \quad \ell=1,\dots,M. \label{not3}
\end{eqnarray}
In this setting, we consider the Gram matrix $\mathcal{A}_m\in\mathbb{R}^{M\times M}$, defined as 
\begin{equation*}
\left( \mathcal{A}_m \right)_{j,\ell} = Z_{\ell}\left(\boldsymbol{p}_j\right), \quad j,\ell=1,2,\dots,M. 
\end{equation*}
This matrix is used to quantify the stability of the nodes via its condition number 
\begin{equation*}
\kappa_{\infty}\left(\mathcal{A}_m\right)=\left\lVert \mathcal{A}_m\right\rVert_{\infty} \left\lVert \mathcal{A}_m^{-1}\right\rVert_{\infty}.
\end{equation*}
Since the radii are the main parameters affecting the conditioning of $\mathcal{A}_m$, we consider the following minimization problem
\begin{equation*}
    \min \left\{ \kappa_{\infty}\left(\mathcal{A}_m\right) \, :\, 1 > \rho_1 > \rho_2 > \cdots > \rho_K \ge 0 \right\}.
\end{equation*}
Although no explicit formula for the optimal radii exists, numerical results suggest a good approximation via least squares fitting~\cite{cuyt2012radial, carnicer2014interpolation}. Specifically, the radii can be approximated by
\begin{equation*}
\rho_\nu^{\mathrm{opt}} = \rho_\nu^{\mathrm{opt}}(m) = 1.1565 \zeta_{\nu,m} - 0.76535 \zeta_{\nu,m}^2 + 0.60517 \zeta_{\nu,m}^3,    
\end{equation*}
where $\zeta_{\nu,m}$ are the zeros of the  Chebyshev polynomial of the first kind of degree $m+1$, that is
\begin{equation*}
    \zeta_{\nu,m} = \cos\left(\frac{(2\nu-1)\pi}{2(m+1)}\right), \quad \nu = 1,2,\dots,K.
\end{equation*}
Once the radii are fixed, the \textit{optimal nodes} are given by
\begin{equation}\label{optpoints}
\boldsymbol{p}_{\nu,\sigma}^{\mathrm{opt}} = \left(\rho_\nu^{\mathrm{opt}}\cos\left(\frac{2\pi (\sigma-1)}{m_\nu}\right), \rho_\nu^{\mathrm{opt}}\sin\left(\frac{2\pi (\sigma-1)}{m_\nu}\right) \right), \quad \sigma=1,\dots, m_\nu,\quad \nu=1,\dots,K.    
\end{equation}
We denote the set of optimal nodes by
\begin{equation*}
    \Sigma_m^{\mathrm{opt}}=\left\{\boldsymbol{p}_{\nu,\sigma}^{\mathrm{opt}}\, : \, \sigma=1,\dots, m_\nu,\, \nu=1,\dots,K \right\},
\end{equation*}
see Fig~\ref{figrec1}.
In principle, if $f$ were known over the whole domain $D_1$, one could construct an accurate Zernike polynomial interpolant using the optimal nodes of the set $\Sigma_m^{\mathrm{opt}}$. However, in practical applications, $f$ is typically sampled on a general grid $S_{n,k}$, which often reflects measurements of a physical phenomenon and thus precludes the direct application of this strategy. Motivated by this limitation, we propose an interpolation-regression approach that exploits all available data to construct an accurate polynomial approximation of $f$.

\begin{figure}
  \centering
\includegraphics[width=0.49\textwidth]{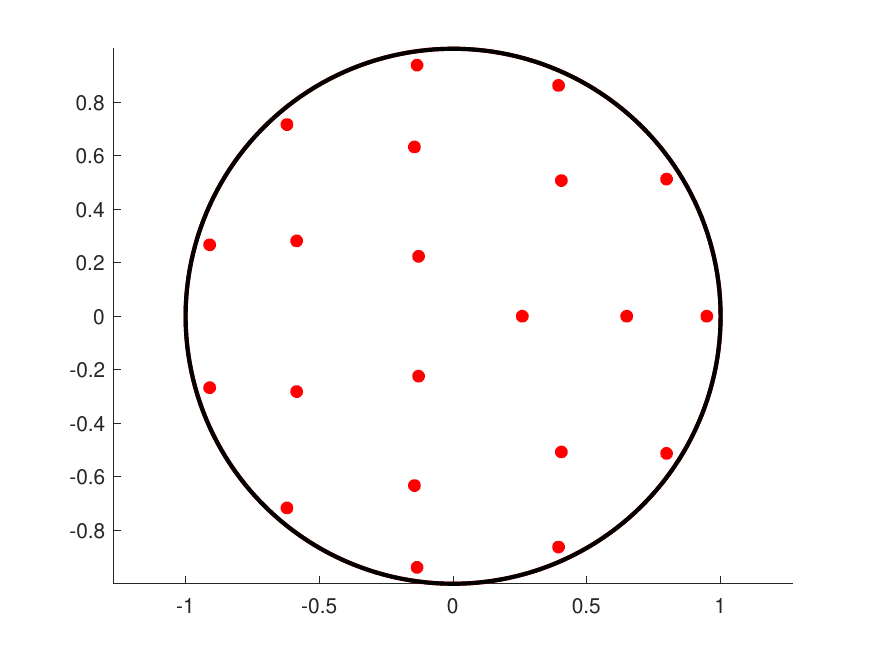} 
\includegraphics[width=0.49\textwidth]{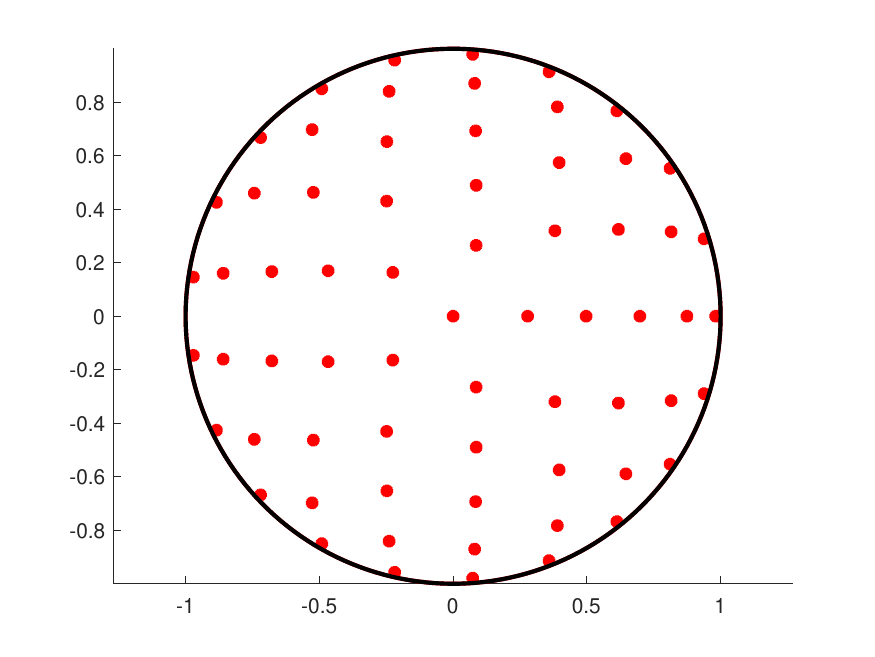} 
\includegraphics[width=0.49\textwidth]{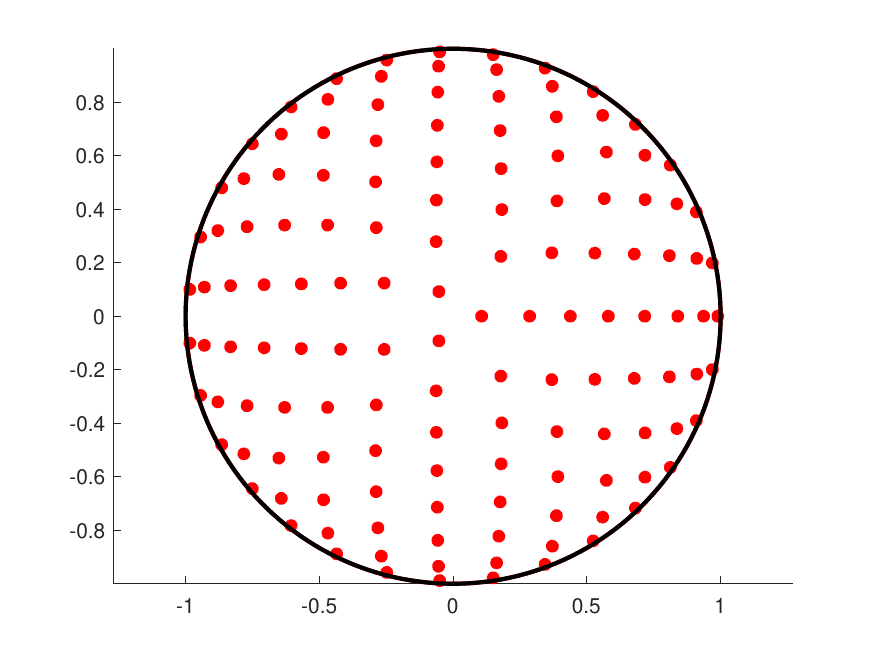}
\includegraphics[width=0.49\textwidth]{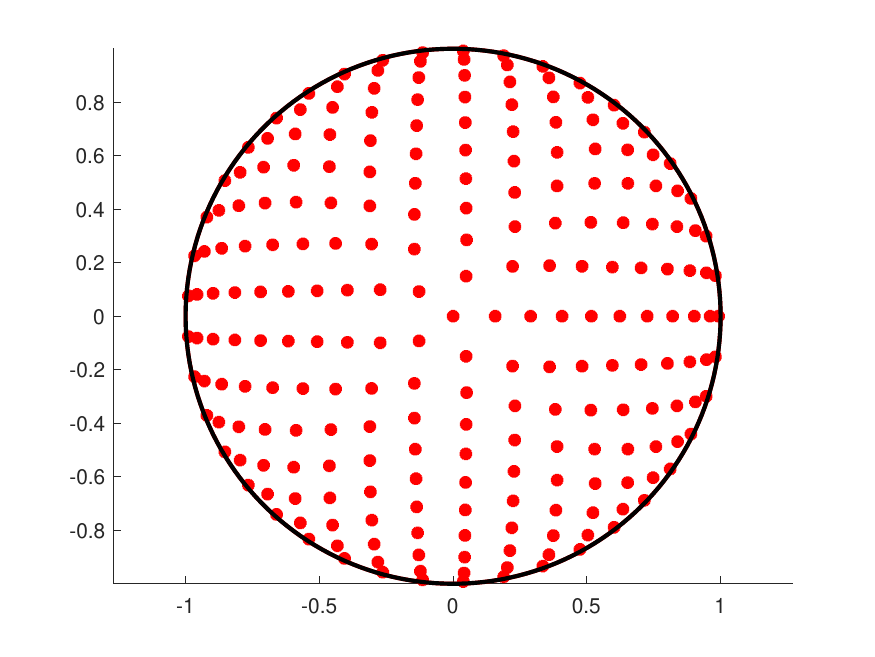}
 \caption{Plot of the optimal nodes of the set $\Sigma_m^{\mathrm{opt}}$, for $m=5,10,15,20$.}
 \label{figrec1}
\end{figure}

\subsection{Interpolation-regression operator on the disk}

Let $f\in C\left(D_1,\mathbb{R}\right)$ be an unknown continuous real-valued function defined on the unit disk $D_1$, and assume that its evaluations are available only on the set $S_{n,k}$ defined in~\eqref{pointset}. For simplicity, we consider the case in which $k=n$ and denote by
\begin{equation*}
n^{\star}=\#\left(S_{n,n}\right)=(n+1)^2. 
\end{equation*}

\noindent
In order to define our interpolation-regression operator, some settings are needed. Let $m,\tilde{r}\in\mathbb{N}$ be such that
 \begin{equation}\label{mpdef}
     \tilde{r}>m, \quad \tilde{R}:=\frac{(\tilde{r}+1)(\tilde{r}+2)}{2}<n^{\star}, \quad M:=\frac{(m+1)(m+2)}{2},
 \end{equation}
and consider a basis of the polynomial space $\mathbb{P}_{\tilde{r}}\left(\mathbb{R}^2\right)$
\begin{equation}\label{basis}
    \mathcal{B}_{\tilde{r}}=\left\{u_1,\dots,u_{\tilde{R}}\right\}, 
\end{equation} 
ensuring that
\begin{equation}
    \operatorname{span}\left\{u_1,\dots,u_M\right\}=\mathbb{P}_m\left(\mathbb{R}^2\right). 
\end{equation}
Using the notation in~\eqref{not1}, we define the matrix  
\begin{equation}\label{matrixVandC}
  \mathcal{M}:=\begin{bmatrix}
u_1\left(\boldsymbol{x}_1\right) & u_2\left(\boldsymbol{x}_1\right) & \dots & u_{\tilde{R}}\left(\boldsymbol{x}_1\right)\\
u_1\left(\boldsymbol{x}_2\right) & u_2\left(\boldsymbol{x}_2\right) & \dots & u_{\tilde{R}}\left(\boldsymbol{x}_2\right)\\
\vdots  & \vdots  & \dots & \vdots  \\
u_1\left(\boldsymbol{x}_{n^{\star}}\right) & u_2\left(\boldsymbol{x}_{n^{\star}}\right) & \dots & u_{\tilde{R}}\left(\boldsymbol{x}_{n^{\star}}\right)\\
\end{bmatrix}\in\mathbb{R}^{n^{\star}\times {\tilde{R}}},
\end{equation}
and assume that the set of samplings $S_{n,n}$ satisfies 
\begin{equation}\label{assump1}
\operatorname{rank}\left(\mathcal{M}\right)=\tilde{R}.
\end{equation}
In this case, the nodes of $S_{n,n}$ are said to be \textit{admissible} with respect to the degree $\tilde{r}$.

\begin{remark}
We note that the assumption~\eqref{assump1} is not restrictive in practice. In fact, if the set $S_{n,n}$ does not satisfy it for a fixed $\tilde{r}$, one can either increase the number of sampling points or decrease the degree $\tilde{r}$ until the assumption is satisfied.
\end{remark}

\noindent
Our goal is to introduce a polynomial approximation operator that employs all nodes of $S_{n,n}$ while leveraging, in a certain sense, the optimality of the points of $\Sigma_{m}^{\mathrm{opt}}$. \\

\noindent 
The idea is to construct this approximation so that: 
\begin{itemize}
    \item[$i)$] it interpolates the unknown function $f$ on an \emph{optimal subset} of $S_{n,n}$,  consisting of the nodes closest to those of $\Sigma_{m}^{\mathrm{opt}}$;
    \item[$ii)$] it exploits the remaining nodes to enhance the approximation accuracy through a simultaneous regression. 
\end{itemize}
In analogy to the constrained mock-Chebyshev least squares operator on square domains~\cite{DellAccio:2022:GOT, DellAccio:2024:AEO} and on triangular domains~\cite{de2024mixed}, we consider the set of \textit{mock-optimal} nodes
\begin{equation*}
\Sigma_m^{\prime}:=\left\{\boldsymbol{x}^{\prime}_{\nu,\sigma} \, : \, \sigma=1,\dots,m_\nu, \, \nu=1,\dots,K\right\}\subset S_{n,n},    
\end{equation*}
where $\boldsymbol{x}^{\prime}_{\nu,\sigma}$ is the solution of the minimization problem
\begin{equation}\label{minprobs}
\min_{\boldsymbol{x}_{\eta,\kappa}\in S_{n,n}}\left\lVert \boldsymbol{x}_{\eta,\kappa}-\boldsymbol{p}_{\nu,\sigma}^{\mathrm{opt}} \right\rVert,
\end{equation}
where $\left\lVert \cdot\right\rVert$ is the Euclidean norm in $\mathbb{R}^2$. \\

\noindent
Specifically, for each optimal node, we select the closest point from the set $S_{n,n}$; if that point has already been selected, we choose the next closest one, and so on. The Algorithm~\ref{mock-det} formalizes this selection process and   ensures that
\begin{equation}\label{assump2}
    \#\left(\Sigma_m^{\prime}\right)=M.
\end{equation}

\begin{algorithm}
\caption{Selection of mock-optimal nodes}
\begin{algorithmic}[1]
\STATE $n,m\in\mathbb{N}$
\STATE $S_{n,n} \gets \left[\boldsymbol{x}_{\eta ,\kappa}(:,1),\, \boldsymbol{x}_{\eta ,\kappa}(:,2)\right]$
\STATE $\Sigma_m^{\text{opt}} \gets \left[\boldsymbol{p}_{\nu,\sigma}^{\mathrm{opt}}(:,1),\boldsymbol{p}_{\nu,\sigma}^{\mathrm{opt}}(:,2)\right]$
\STATE $S_{n,n}^{\text{copy}} \gets S_{n,n}$
\STATE $\Sigma_m^{\prime} \gets \text{zeros}\left(\text{size}\left(\Sigma_m^{\text{opt}}\right)\right)$
\FOR{$j = 1$ to $M$}
    \STATE \(d_{1} \gets \left|S_{n,n}^{\text{copy}}(:,1) - \Sigma_m^{\text{opt}}(j,1)\right|\)
    \STATE \(d_{2} \gets \left|S_{n,n}^{\text{copy}}(:,2) - \Sigma_m^{\text{opt}}(j,2)\right|\)
    \STATE \(d \gets d_{1} + d_{2}\)
    \STATE \(i \gets \text{argmin}(d)\)
    \STATE \(\Sigma_{m}^{\prime}(j,:) \gets S_{n,n}^{\text{copy}}(i,:)\)
    \STATE Remove the \(i\)-th row from \(S_{n,n}^{\text{copy}}\)
\ENDFOR
\end{algorithmic}
\label{mock-det}
\end{algorithm}

\noindent
Using the notation~\eqref{not2}, we  consider the Gram matrix
\begin{equation}\label{matrixC}
\mathcal{C}:=\begin{bmatrix}
u_1\left(\boldsymbol{x}^{\prime}_1\right) & u_2\left(\boldsymbol{x}^{\prime}_1\right) & \dots & u_{\tilde{R}}\left(\boldsymbol{x}^{\prime}_1\right)\\
u_1\left(\boldsymbol{x}^{\prime}_2\right) & u_2\left(\boldsymbol{x}^{\prime}_2\right) & \dots & u_{\tilde{R}}\left(\boldsymbol{x}^{\prime}_2\right)\\
\vdots  & \vdots  & \dots & \vdots  \\
u_1\left(\boldsymbol{x}^{\prime}_{M}\right) & u_2\left(\boldsymbol{x}^{\prime}_{M}\right) & \dots & u_{\tilde{R}}\left(\boldsymbol{x}^{\prime}_{M}\right)\\
\end{bmatrix}\in\mathbb{R}^{M\times {\tilde{R}}}
\end{equation}
and the vectors
\begin{equation}\label{bandd}
\boldsymbol{b}:=\left[f\left(\boldsymbol{x}_1\right),\dots,f\left(\boldsymbol{x}_{n^{\star}}\right)\right]^T, \quad  \boldsymbol{d}:=\left[f\left(\boldsymbol{x}_1^{\prime}\right),\dots,f\left(\boldsymbol{x}_{M}^{\prime}\right)\right]^T.
\end{equation}
We observe that if the nodes of $\Sigma_m^{\prime}$ are sufficiently close to those of the set $\Sigma_m^{\mathrm{opt}}$, then the matrix $\mathcal{C}$ satisfies~\cite{viannuovo}
\begin{equation*}
\operatorname{rank}(\mathcal{C}) = M.
\end{equation*}

\noindent
We then define the interpolation-regression operator
\begin{equation}\label{operatorPhat}
\hat{\Pi}_{{\tilde{r}},n}: f\in C\left(D_1,\mathbb{R}\right) \rightarrow \hat{\Pi}_{\tilde{r},n}[f]:=\sum\limits_{i=1}^{\tilde{R}}\hat{a}_iu_i\in \mathbb{P}_{\tilde{r}}\left(\mathbb{R}^2\right),
\end{equation}
where the vector of coefficients $\hat{\boldsymbol{a}}:=\left[\hat{a}_1,\dots,\hat{a}_{\tilde{R}}\right]^T$ can be computed solving the so-called KKT linear system~\cite{Boyd:2018:ITA}
\begin{equation}\label{LagrangeMultmethod}
    \begin{bmatrix}
\mathcal{W} & \mathcal{C}^T  \\
\mathcal{C} & 0  \\
\end{bmatrix}
\begin{bmatrix}
\hat{\boldsymbol{a}} \\
\hat{\boldsymbol{z}} \\
\end{bmatrix}=
\begin{bmatrix}
2 \mathcal{M}^T\boldsymbol{b} \\
\boldsymbol{d} \\
\end{bmatrix}, \quad  \mathcal{W}:=2 \mathcal{M}^T\mathcal{M},
\end{equation}
and $\hat{\boldsymbol{z}}:=\left[\hat{z}_1,\dots,\hat{z}_{M}\right]^T$ is the Lagrange multipliers vector. This system computes $\hat{\boldsymbol{a}}$ as the solution of
the following least squares problem with linear equality constraints
\begin{equation}
\label{LSProblem}
\min_{\boldsymbol{a}\in \mathbb{R}^{{\tilde{R}}}}\left\lVert \mathcal{M}\boldsymbol{a}-\boldsymbol{b}\right\rVert_2, \quad \text{subject to} \quad  \mathcal{C} \hat{\boldsymbol{a}} = \boldsymbol{d},
\end{equation}
thereby enforcing properties $i)$ and $ii)$. The matrix
\begin{equation}
    \label{KKTmatrix}
  \mathcal{K}= \begin{bmatrix}
\mathcal{W} & \mathcal{C}^T  \\
\mathcal{C} & 0  \\
\end{bmatrix}
\end{equation}
is called KKT matrix. \\

\begin{remark}
Since both $\mathcal{M}$ and $\mathcal{C}$ have full rank (cf. \eqref{assump1} and \eqref{assump2}), the matrix $\mathcal{K}$ defined in~\eqref{KKTmatrix} is nonsingular~\cite{Boyd:2018:ITA}. Therefore, the vector of coefficients $\hat{\boldsymbol{a}}$ is uniquely determined and then the operator $\hat{\Pi}_{\tilde{r},n}$ is well defined.
\end{remark}

\begin{remark}
The operator $\hat{\Pi}_{\tilde{r},n}$ has the following properties:
\begin{itemize}
    \item[-] It is linear, i.e.,
    \[
    \hat{\Pi}_{\tilde{r},n}[\lambda f+\mu g]=\lambda \hat{\Pi}_{\tilde{r},n}[f] + \mu \hat{\Pi}_{\tilde{r},n}[g],\quad f,g\in C\left(D_1,\mathbb{R}\right),\ \lambda,\mu\in\mathbb{R};
    \]
    \item[-] Its range is $\mathbb{P}_{\tilde{r}}(\mathbb{R}^2)$;
    \item[-] It reproduces bivariate polynomials of degree $\le\tilde{r}$, i.e.,
    \begin{equation*}
         \hat{\Pi}_{\tilde{r},n}[q]=q, \text{ for each } q\in\mathbb{P}_{\tilde{r}}\left(\mathbb{R}^2\right);
    \end{equation*}
    \item[-] It is idempotent, that is, $\hat{\Pi}_{\tilde{r},n}\circ \hat{\Pi}_{\tilde{r},n}=\hat{\Pi}_{\tilde{r},n}$;
    \item[-] $\hat{\Pi}_{\tilde{r},n}[f]$ is completely determined by the evaluations of $f$ at the nodes of $S_{n,n}$.
\end{itemize}
\end{remark}

\begin{remark}
All our reasoning has been carried out for point sets $S_{n,k}$ consisting of points that can be uniquely expressed in polar coordinates. This assumption has been made only for simplicity. In fact, the same arguments can be repeated when the point set includes the origin, that is
\begin{equation*}
    \tilde{S}_{n,k}=S_{n,k}\cup \{\boldsymbol{0}=(0,0)\}.
\end{equation*}
In this case, we get
\begin{equation*}
    \#\left(\tilde{S}_{n,k}\right)=(n+1)(k+1)+1.
\end{equation*}
\end{remark}

\noindent
In order to present an error bound for the operator~\eqref{operatorPhat}, we recall the \emph{direct elimination method} for solving least squares problems with linear equality constraints~\cite{bjork1967iterative, bjorck2024numerical}. \\

\noindent
Using the QR factorization of $\mathcal{C}$, it results
\begin{equation*}
    \mathcal{C} = \mathcal{Q}_\mathcal{C} \mathcal{R}_\mathcal{C},\quad \mathcal{Q}_\mathcal{C} \in \mathbb{R}^{M\times M}, \quad \mathcal{Q}_\mathcal{C}\mathcal{Q}_\mathcal{C}^T=\mathcal{I}
\end{equation*}
and applying $\mathcal{Q}_\mathcal{C}^T$ to the constraints equation yields
\begin{equation}\label{sistema_intermedio}
\mathcal{R}_{\mathcal{C}} \hat{\boldsymbol{a}} = \mathcal{Q}_{\mathcal{C}}^T \boldsymbol{d}.
\end{equation}
We decompose $\mathcal{R}_\mathcal{C}$ as
\begin{equation*}
\mathcal{R}_{\mathcal{C}}=\left[\mathcal{R}_{11},\mathcal{R}_{12}\right], \quad \mathcal{R}_{11}\in\mathbb{R}^{M\times M}, \quad \mathcal{R}_{12}\in\mathbb{R}^{M\times \left(\tilde{R}-M\right)}. 
\end{equation*}
Then, writing $\hat{\boldsymbol{a}}=\begin{bmatrix}\hat{\boldsymbol{a}}_1\\ \hat{\boldsymbol{a}}_2\end{bmatrix}$ with $\hat{\boldsymbol{a}}_1\in\mathbb{R}^{M}$ and $\hat{\boldsymbol{a}}_2\in\mathbb{R}^{\tilde{R}-M}$, the linear system~\eqref{sistema_intermedio} becomes
\begin{equation*}
\left[\mathcal{R}_{11},\mathcal{R}_{12}\right]
\begin{bmatrix}
\hat{\boldsymbol{a}}_1 \\
\hat{\boldsymbol{a}}_2
\end{bmatrix}
= \mathcal{Q}_\mathcal{C}^T\boldsymbol{d}.
\end{equation*}
Since $\mathcal{R}_{11}$ is nonsingular, we can write
\begin{equation}\label{a1eq}
\hat{\boldsymbol{a}}_1 = \mathcal{R}_{11}^{-1}\left(\mathcal{Q}_\mathcal{C}^T\boldsymbol{d} - \mathcal{R}_{12}\hat{\boldsymbol{a}}_2\right).
\end{equation}
Substituting $\hat{\boldsymbol{a}}_1$ into the original least squares problem leads to an unconstrained problem in terms of $\hat{\boldsymbol{a}}_2$. In fact, by setting
\begin{equation*}
  \mathcal{M} = \left[\mathcal{M}_1, \mathcal{M}_2\right], \quad \mathcal{M}_1\in\mathbb{R}^{n^{\star}\times M}, \quad \mathcal{M}_{2}\in\mathbb{R}^{n^{\star}\times \left(\tilde{R}-M\right)},
\end{equation*}
the problem~\eqref{LSProblem} reduces to the unconstrained minimization problem
\begin{equation*}
      \min_{\boldsymbol{a}_2\in \mathbb{R}^{\tilde{R}-M}} \left\| \mathcal{V}_1 \boldsymbol{a}_2 - \boldsymbol{b}_1\right\|_2,
\end{equation*}
where
\begin{equation}\label{M1}
    \mathcal{V}_1=\mathcal{M}_2-\mathcal{M}_1\mathcal{R}_{11}^{-1}\mathcal{R}_{12}\in \mathbb{R}^{n^{\star}\times \left(\tilde{R}-M\right)}
\end{equation}
and
\begin{equation}\label{B1}
\boldsymbol{b}_1 = \boldsymbol{b} - \mathcal{M}_1 \mathcal{R}_{11}^{-1} \mathcal{Q}_{\mathcal{C}}^T\boldsymbol{d}\in \mathbb{R}^{n^{\star}}.
\end{equation}
The solution is then given by
\begin{equation}\label{a2eq}
\hat{\boldsymbol{a}}_2=\left(\mathcal{V}_1^{T}\mathcal{V}_1\right)^{-1}\mathcal{V}_1^{T}\boldsymbol{b}_1.
\end{equation}
This approach yields the solution of problem~\eqref{LSProblem}; for more information, see~\cite{bjorck2024numerical}.

\noindent
In the following theorem, we obtain an upper bound of the operator norm
\begin{equation*}
\left\lVert\hat{\Pi}_{\tilde{r},n}\right\rVert_\infty:=\sup_{\left\lVert f\right\rVert_\infty=1} \left\lVert\hat{\Pi}_{\tilde{r},n}[f]\right\rVert_\infty.
\end{equation*}

\begin{theorem}
\label{Thm:EstimateOfNorm}
The norm of the interpolation-regression operator~\eqref{operatorPhat} satisfies
\begin{equation}
\label{eq:BoundThm}
    \left\lVert \hat{\Pi}_{\tilde{r},n} \right\rVert_{\infty} \le \max_{i=1,\dots,\tilde{R}}\left\lVert u_i \right\rVert_{\infty} \left(K_1(n,m)+K_2(n,m)\right)
\end{equation}
where
\begin{equation}\label{K2}
K_1(n,m)=\left\lVert \mathcal{R}_{11}^{-1} \right\rVert_{1} \left(M\left\lVert \mathcal{Q}_{\mathcal{C}}^{T} \right\rVert_{1}+ \left\lVert \mathcal{R}_{12}\right\rVert_{1}K_2(n)\right),
\end{equation}
\begin{equation}\label{K21}
K_2(n,m)=\left\lVert(\mathcal{V}_1^T\mathcal{V}_1)^{-1} \mathcal{V}_1^T \right\rVert_1\left(n^{\star}+M\left\lVert  \mathcal{M}_1 \mathcal{R}_{11}^{-1}\mathcal{Q}_{\mathcal{C}}^T\right\rVert_1\right).
\end{equation}
\end{theorem}

\begin{proof}
Let $f\in C\left(D_1,\mathbb{R}\right)$ be such that   
\begin{equation} \label{hyp}
\left\lVert f \right\rVert_{\infty}=\max_{\boldsymbol{x}\in D_1}\left\lvert f(\boldsymbol{x})\right\rvert=1.   
\end{equation}
For any $\boldsymbol{x}\in D_1$, we get
\begin{equation*}
    \left|\hat{\Pi}_{\tilde{r},n}[f](\boldsymbol{x})\right|= \left|\sum_{i=1}^{\tilde{R}} \hat{a}_iu_i(\boldsymbol{x})\right| \le \sum_{i=1}^{\tilde{R}}\left|\hat{a}_iu_i(\boldsymbol{x})\right|= \sum_{i=1}^{\tilde{R}} \left\lvert\hat{a}_i\right\rvert \left\lvert u_i(\boldsymbol{x})\right\rvert.
\end{equation*}
Thus we obtain
\begin{equation*}
\left\lVert\hat{\Pi}_{\tilde{r},n}[f]\right\rVert_{\infty}\le  \max_{i=1,\dots,\tilde{R}}\left\lVert u_i \right\rVert_{\infty} \sum_{i=1}^{\tilde{R}} \left|\hat{a}_i\right|=\max_{i=1,\dots,\tilde{R}}\left\lVert u_i \right\rVert_{\infty} \left\lVert\hat{\boldsymbol{a}}\right\rVert_1. 
\end{equation*}
To complete the proof, we bound the $L_1$-norm of $\hat{\boldsymbol{a}}=\hat{\boldsymbol{a}}(f)$. Writing 
\begin{eqnarray*}
\left\lVert
\hat{\boldsymbol{a}} 
\right\rVert_1
&=& 
\left\lVert
\begin{bmatrix}
\hat{\boldsymbol{a}}_1 \\
\hat{\boldsymbol{a}}_2
\end{bmatrix}
\right\rVert_1
= \lVert \hat{\boldsymbol{a}}_1 \rVert_1 + \lVert \hat{\boldsymbol{a}}_2 \rVert_1,
\end{eqnarray*}
and using~\eqref{a2eq} we obtain
\begin{equation}\label{a2}
   \left\lVert \hat{\boldsymbol{a}}_2\right\rVert_1\le\left\lVert\left(\mathcal{V}_1^{T}\mathcal{V}_1\right)^{-1}\mathcal{V}_1^{T}\right\rVert_1 \left\lVert \boldsymbol{b}_1\right\rVert_1.
\end{equation}
From~\eqref{B1}
\begin{equation*}
\left\|\boldsymbol{b}_1\right\|_1 \le \left \|\boldsymbol{b}\right\|_1 +\left\| \mathcal{M}_1 \mathcal{R}_{11}^{-1} \mathcal{Q}_{\mathcal{C}}^T\right\|_1 \left\|\boldsymbol{d}\right\|_1\le  n^{\star} + M\left\| \mathcal{M}_1 \mathcal{R}_{11}^{-1} \mathcal{Q}_{\mathcal{C}}^T\right\|_1.
\end{equation*}
Thus, by~\eqref{a2}, it holds
 \begin{equation}
 \label{aa2}
     \left\lVert \hat{\boldsymbol{a}}_2\right\rVert_1\le   \left\lVert(\mathcal{V}_1^{T}\mathcal{V}_1)^{-1}\mathcal{V}_1^{T}\right\rVert_1 \left( n^{\star} + M\left\|\mathcal{M}_1 \mathcal{R}_{11}^{-1} \mathcal{Q}_{\mathcal{C}}^T\right\|_1 \right).
 \end{equation}
Similarly, from~\eqref{a1eq} we have
 \begin{equation*}
    \left\|\hat{\boldsymbol{a}}_1\right\|_1 \le   \left\| \mathcal{R}_{11}^{-1}\right\|_1 \left(M\left\|\mathcal{Q}_\mathcal{C}^T\right\|_1 + \left\|R_{12}\right\|_1  \left\lVert \hat{\boldsymbol{a}}_2\right\rVert_1\right).
 \end{equation*}
Combining these estimates yields the stated upper bound.
\end{proof}

Since the operator~\eqref{operatorPhat} is a projection on $\mathbb{P}_{\tilde{r}}\left(\mathbb{R}^2\right)$~\cite[Ch. 6]{Cheney:1998:ITA}, standard estimates for the uniform norm of the approximation error
\begin{equation}
\label{eq:ApproximationError}
   \left\lVert f-\hat{\Pi}_{\tilde{r},n}[f]\right\rVert_{\infty}, \quad f\in C\left(D_1,\mathbb{R}\right),
\end{equation}
can be given in terms of the error of best uniform approximation by polynomials in $\mathbb{P}_{\tilde{r}}\left(\mathbb{R}^2\right)$,
\begin{equation}\label{errorbestunifapp}
    \mathcal{E}_{\tilde{r}}(f):=\left\lVert f-p^{\star}_{\tilde{r}}[f]\right\rVert_{\infty},
\end{equation}
also known as the minimax error. With this in mind, the following theorem holds.
\begin{theorem}
\label{Thm:Errorb}
Let $f\in C\left(D_1,\mathbb{R}\right)$, then 
\begin{equation}   \label{bounderroruniformapprox}
    \left\lVert f-\hat{\Pi}_{\tilde{r},n}[f]\right\rVert_{\infty}\le \left(1+\left\lVert u_i \right\rVert_{\infty} \left(K_1(n,m)+K_2(n,m)\right)\right)\mathcal{E}_{\tilde{r}}(f).
\end{equation}
\end{theorem}
\begin{proof}
Let $p^{\star}_{\tilde{r}}[f]\in \mathbb{P}_{\tilde{r}}\left(\mathbb{R}^2\right)$ be the polynomial of best uniform approximation to $f$. Then,
\begin{eqnarray*}
    \left\lVert f-\hat{\Pi}_{\tilde{r},n}[f]\right\rVert_{\infty}&=&    \left\lVert f-p_{\tilde{r}}^{\star}[f]+p_{\tilde{r}}^{\star}[f]-\hat{\Pi}_{\tilde{r},n}[f]\right\rVert_{\infty}\\
    &=&\left\lVert f-p_{\tilde{r}}^{\star}[f]-\hat{\Pi}_{\tilde{r},n}\left[f-p_{\tilde{r}}^{\star}[f]\right]\right\rVert_{\infty}\\
    &\le& \left(1+\left\lVert \hat{\Pi}_{\tilde{r},n}\right\rVert_{\infty}\right)\left\lVert f-p_{\tilde{r}}^{\star}[f]\right\rVert_{\infty}.
\end{eqnarray*}
Then the statement is a direct consequence of Theorem~\ref{Thm:EstimateOfNorm}.
\end{proof}

\begin{remark}
The error bound in~\eqref{bounderroruniformapprox} strongly depends on the choice of $m$ and $\tilde{r}$ and on the set $S_{n,n}$. Varying these parameters will yield corresponding changes in the upper bound.
\end{remark}

In the following, we present two different examples of sets of sampling points.
\begin{example}
      We consider the set of sampling points
   \begin{equation}\label{pointsetadm}
    \tilde{S}_{n,n}^1=\left\{\boldsymbol{x}_{\eta ,\kappa}=\left(r_\eta\cos \left(\theta_\kappa\right),r_\eta\sin\left(\theta_\kappa\right)\right)\, :\, \eta=0,\dots,n, \, \kappa=0,\dots,n\right\}\cup \{\boldsymbol{0}=(0,0)\}\subset D_1,
\end{equation}
such that
\begin{equation*}
    r_{\eta}=\frac{\eta+1}{{n+1}}, \quad \theta_{\kappa}=\frac{2\pi}{{n+1}}\kappa, \quad \eta,\kappa=0,\dots,n.
\end{equation*}
In this case, we have
\begin{equation*}
    \#\left(\tilde{S}_{n,n}^1\right)=(n+1)^2+1.
\end{equation*}
Let consider  
\begin{equation*}
    m=\left\lfloor \frac{n}{4} \right\rfloor
\end{equation*}
and compute the set of mock-optimal nodes for different values of $n=20,40,60,80$, see Fig.~\ref{figrec2}. Moreover, in Fig.~\ref{figrec3} we show the log-log plot of the bound~\eqref{eq:BoundThm} for the norm of the operator $\hat{\Pi}_{\tilde{r},n}$ with
\begin{equation*}
\tilde{r}=m+\left\lfloor\sqrt{m} \right\rfloor
\end{equation*}
using the Zernike polynomial basis. The graph suggests that the error bound satisfies
     $$K_1(n,m)+K_2(n,m) \approx e^{2.49 \log n}=n^{2.49}.$$
     Moreover, we analyzed the trend of the condition number of the KKT matrix computed for the set $\tilde{S}_{n,n}^{1}$; see Fig.~\ref{condKKT} (left).
\end{example}

\begin{figure}
  \centering
\includegraphics[width=0.49\textwidth]{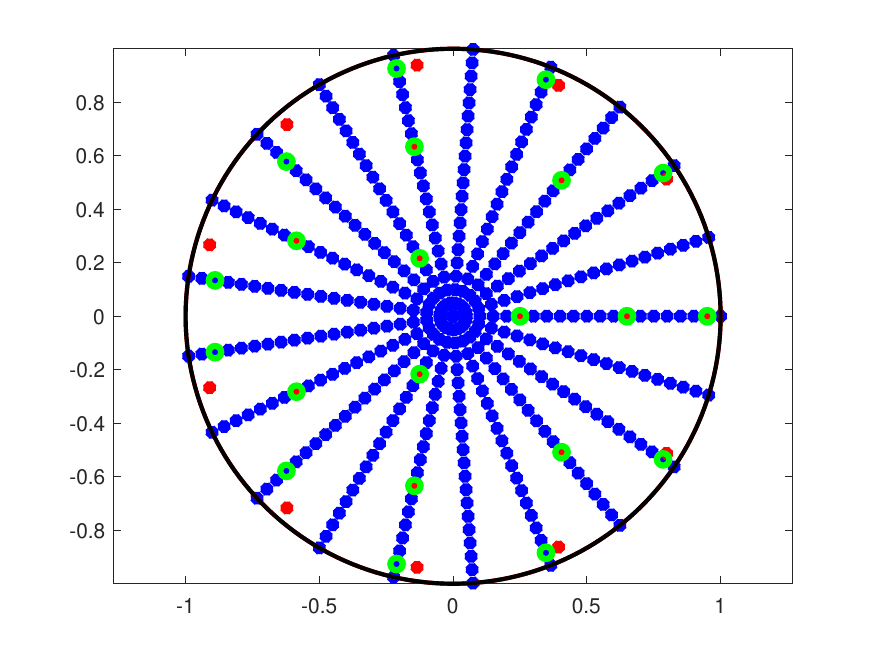} 
\includegraphics[width=0.49\textwidth]{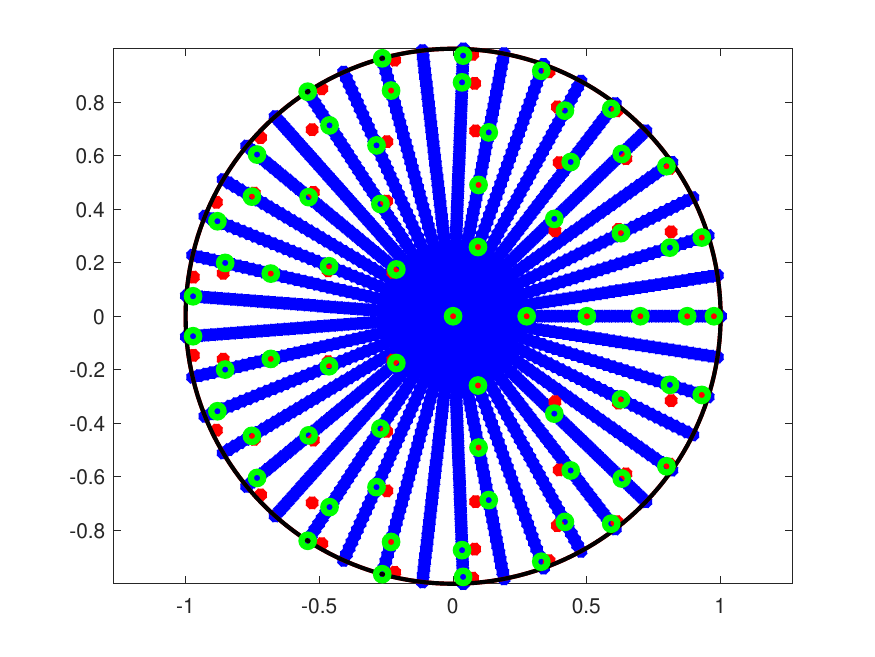} 
\includegraphics[width=0.49\textwidth]{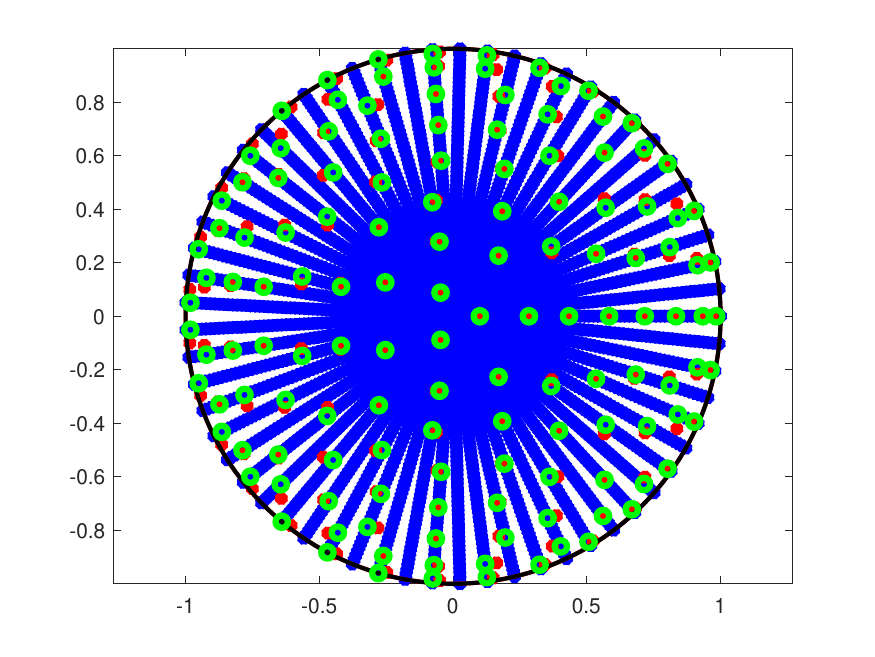}
\includegraphics[width=0.49\textwidth]{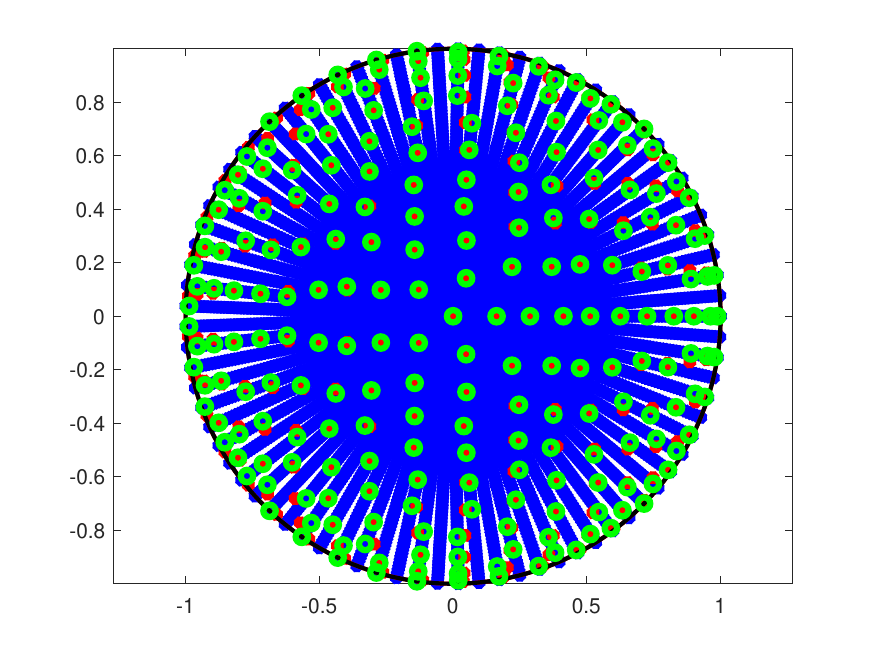}
 \caption{Plot of the sampling nodes of  $\tilde{S}_{n,n}^1$ (\textcolor{blue}{$\star$}), optimal points (\textcolor{red}{$\star$}) and mock-optimal nodes (\textcolor{green}{$\circ$}) for $n=20,40,60,80$, with  $  m=\left\lfloor \frac{n}{4} \right\rfloor$.}
 \label{figrec2}
\end{figure}

\begin{figure}
  \centering
\includegraphics[width=0.49\textwidth]{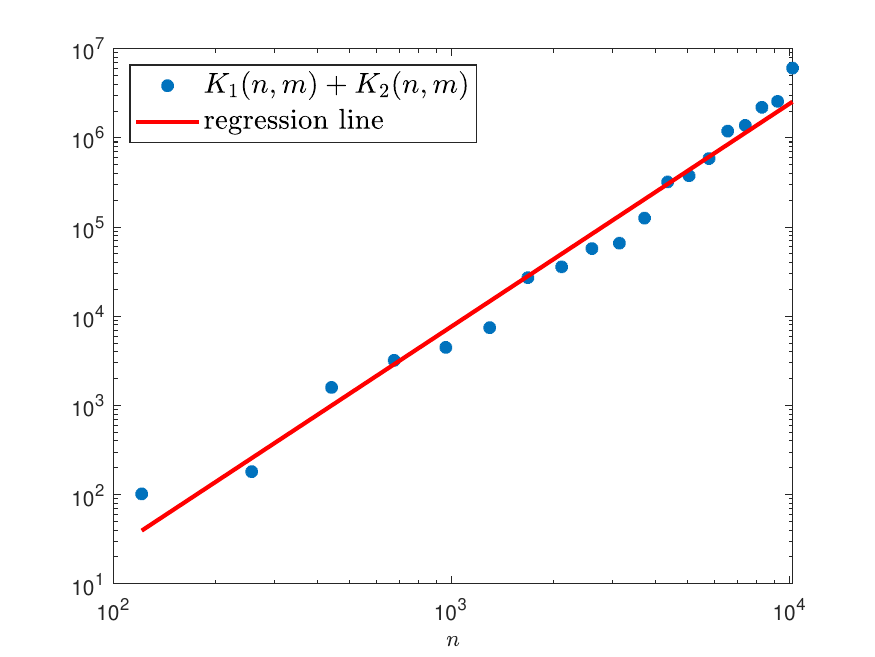} 
 \caption{Trend of the bound for the norm of $\hat{\Pi}_{\tilde{r},n}$ with $  m=\left\lfloor \frac{n}{4} \right\rfloor$, $\tilde{r}=m+\left\lfloor\sqrt{m} \right\rfloor$ varying $n$ from $n=10$ to $n=100$.}
 \label{figrec3}
\end{figure}

\begin{example}
We consider the set of sampling points on the unit disk generated by a spiral arrangement. Let $n^{\star}\in\mathbb{N}$ be the total number of points. For each index
\begin{equation*}
    \iota = 0, 1, \dots, n^{\star}-1,
\end{equation*}
we define the radial and angular coordinates as
\begin{equation*}
    r_{\iota} = \sqrt{\frac{\iota}{n^{\star}}},\quad \theta_{\iota} =\iota\theta_0,
\end{equation*}
where
\begin{equation*}
\theta_0 = \pi\left(3-\sqrt{5}\right) 
\end{equation*}
is the golden angle, which helps to avoid alignment and yields a quasi-uniform distribution in area.
Thus, the set of sampling points is defined as
\begin{equation}\label{spiralnodes}
S^{\mathrm{spiral}}_{n^{\star}} = \left\{\boldsymbol{x}_{\iota}= \left(r_\iota\cos\left(\theta_\iota\right),r_\iota\sin\left(\theta_\iota\right)\right) \, : \, \iota=0,\dots,n^{\star}-1 \right\} \subset D_1.    
\end{equation}
We set $n^{\star}=n^2$ and
\begin{equation*}
    m=\left\lfloor \frac{n}{4} \right\rfloor
\end{equation*}
and we compute the set of mock-optimal nodes for different values of $n=20,40,60,80$, see Fig~\ref{figrec2spi}. Moreover, we analyzed the trend of the condition number of the KKT matrix computed for the set $S_{n^2}^{\mathrm{spiral}}$; see Fig.~\ref{condKKT} (right).
\end{example}

\begin{figure}
  \centering
\includegraphics[width=0.49\textwidth]{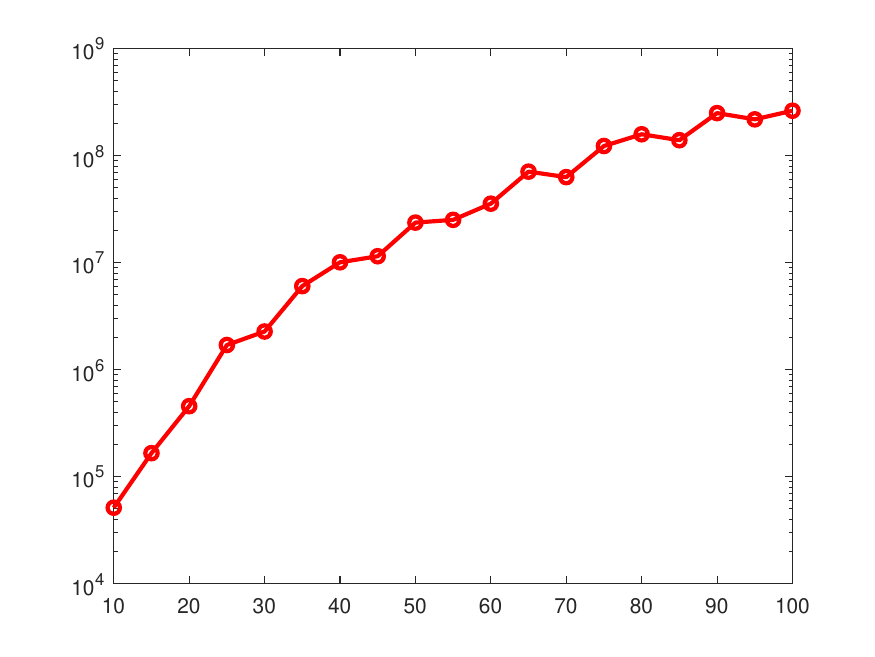}
\includegraphics[width=0.49\textwidth]{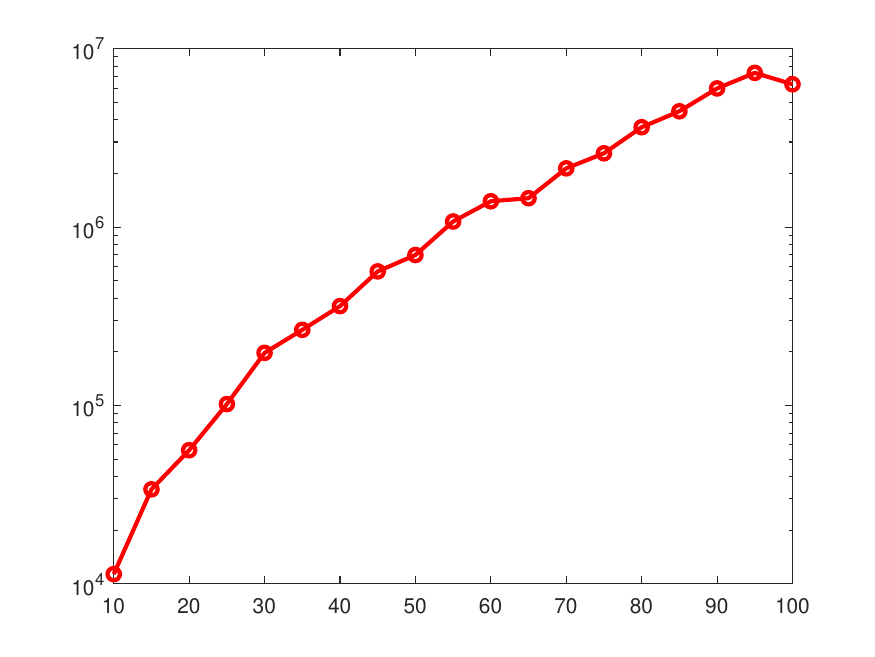}
 \caption{Plot of the condition number of the KKT matrix computed for the set $\tilde{S}^1_{n,n}$ (left) and for the set $S_{n^2}^{\mathrm{spiral}}$ (right), with $n$ varying from $n=10$ to $n=100$, $m = \left\lfloor \frac{n}{4} \right\rfloor$, and $\tilde{r} = m + \left\lfloor \sqrt{m} \right\rfloor$.}
 \label{condKKT}
\end{figure}

\begin{figure}
  \centering
\includegraphics[width=0.49\textwidth]{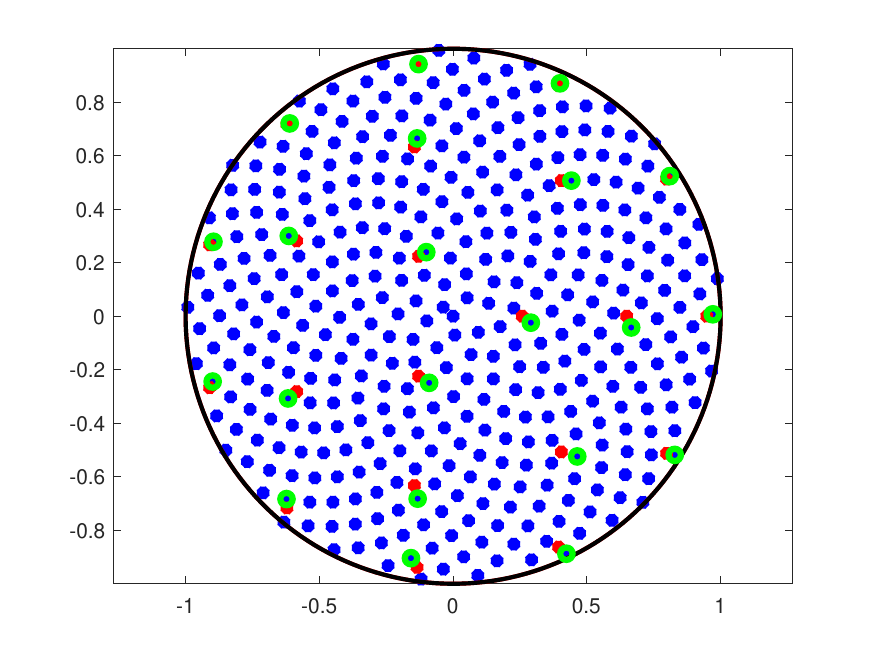} 
\includegraphics[width=0.49\textwidth]{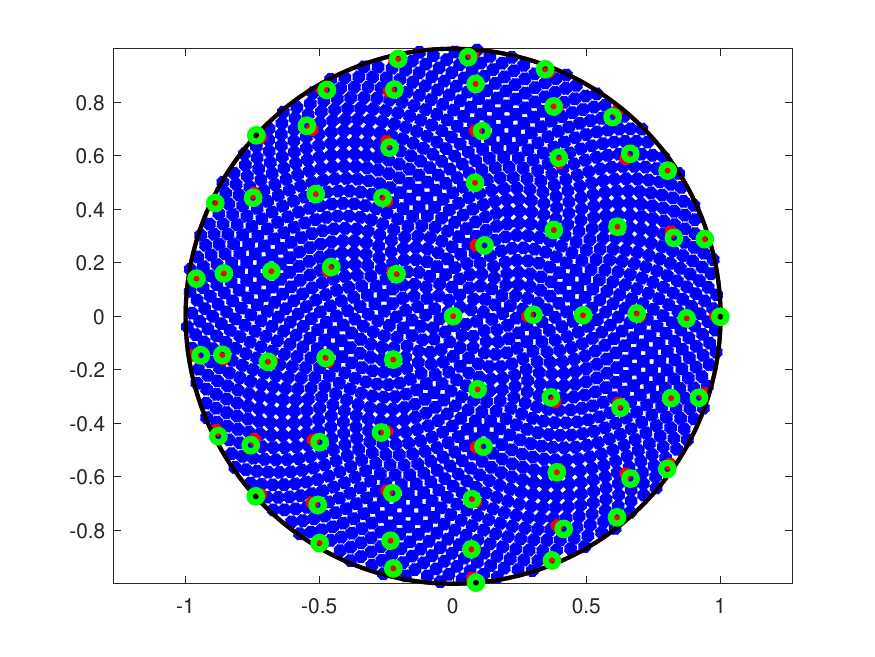} 
\includegraphics[width=0.49\textwidth]{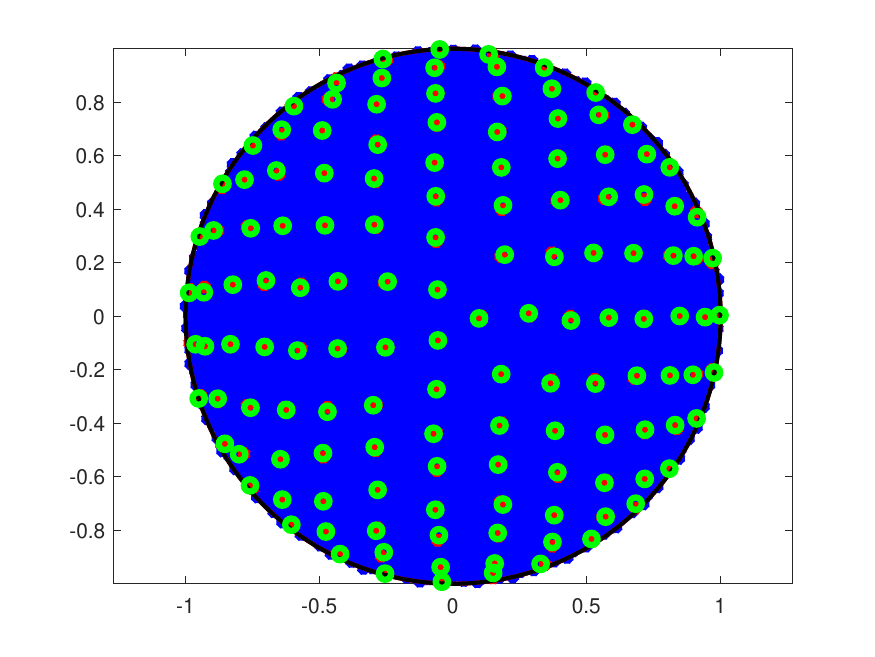}
\includegraphics[width=0.49\textwidth]{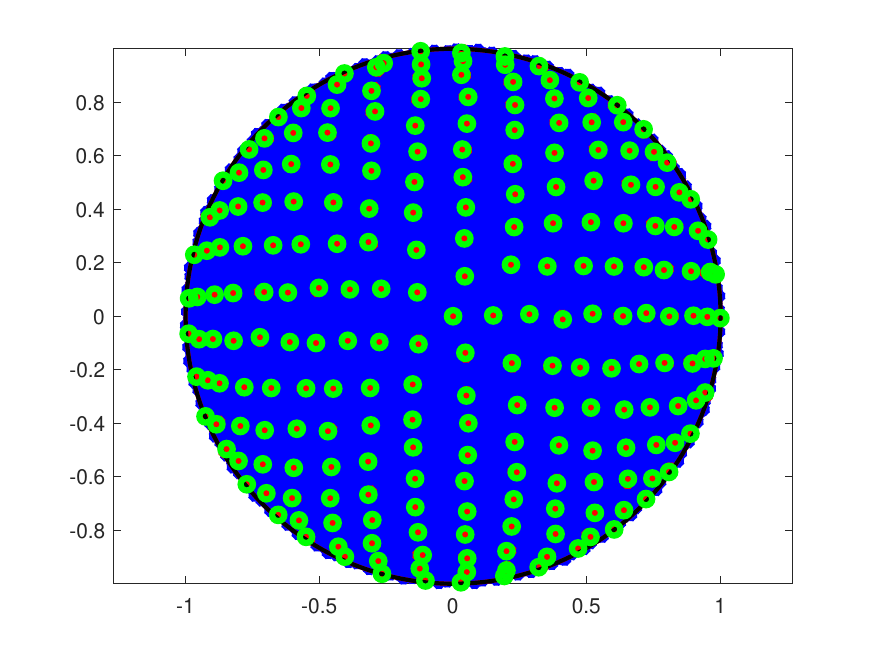}
 \caption{Plot of the sampling nodes of $S^{\mathrm{spiral}}_{n^{\star}}$ (\textcolor{blue}{$\star$}), optimal points (\textcolor{red}{$\star$}) and mock-optimal nodes (\textcolor{green}{$\circ$}) for $n^{\star}=n^2$, $n=20,40,60,80$ with  $  m=\left\lfloor \frac{n}{4} \right\rfloor$.}
 \label{figrec2spi}
\end{figure}

\section{Cubature formula on the disk through an interpolation-regression method}
\label{SecCubFor}
Let $f \in C\left(D_1,\mathbb{R}\right)$ be a continuous function defined on $D_1$, and let $\omega \in L_1\left(D_1\right)$ be a weight function on the disk. Our goal is to introduce a cubature formula to approximate the integral
\begin{equation}\label{inte}
I[f]:=\int_{D_1} f(\boldsymbol{x})\,\omega(\boldsymbol{x})\,d\boldsymbol{x},
\end{equation}
under the assumption that evaluations of $f$ are available only at a prescribed set of nodes $S_{n,k} \subset D_1$. As above, for simplicity, we set $k=n$.

\noindent
If $f$ were known at all points of $D_1$, then a natural idea would be to employ a Gaussian cubature formula on the disk~\cite{kim1997symmetric, bojanov1998numerical, bojanov2000uniqueness}. Unfortunately, in many applications the exact evaluations of $f$ at these points are not available. For this reason, we propose a cubature method based on the interpolation-regression operator defined in~\eqref{operatorPhat}.

\noindent
To implement this strategy, we follow an approach similar to that used in \cite{majidian2016creating, DellAccio:2022:CMC, DellAccio:2022:AAA, dell2024quadrature} for constructing cubature formulas on square and triangular domains. More precisely, let $n, m, \tilde{r} \in \mathbb{N}$ be such that
\begin{equation*}
    \tilde{r}>m, \quad \tilde{R}=\frac{\left(\tilde{r}+1\right)\left(\tilde{r}+2\right)}{2}<n^{\star}=\#\left(S_{n,n}\right).
\end{equation*}
We consider a set of Gaussian cubature nodes and weights on $D_1$
\begin{equation*}
    \{\boldsymbol{\xi}_1,\dots, \boldsymbol{\xi}_P\} \quad\text{and}\quad \{\omega_1,\dots, \omega_P\}, \quad P\in \mathbb{N}.
\end{equation*}
Then, the integral in~\eqref{inte} is approximated by replacing $f$ with its approximation $\hat{\Pi}_{\tilde{r},n}[f]$, that is
\begin{equation}\label{formulaquad}
\int_{D_1} f(\boldsymbol{x})\,\omega(\boldsymbol{x})d\boldsymbol{x} \approx \int_{D_1} \hat{\Pi}_{\tilde{r},n}[f](\boldsymbol{x})\omega(\boldsymbol{x})d\boldsymbol{x} \approx \sum_{j=1}^P\omega_j\hat{\Pi}_{\tilde{r},n}[f]\left(\boldsymbol{\xi}_j\right),
\end{equation}
where the cubature formula
\begin{equation}\label{quadFor}
\hat{Q}_{\tilde{r},n,P}[f] := \sum_{j=1}^P\omega_j\,\hat{\Pi}_{\tilde{r},n}[f]\left(\boldsymbol{\xi}_j\right)
\end{equation}
approximates the integral.

\section{Numerical experiments}
\label{SecNumEx}
In this section, we perform several numerical experiments to demonstrate the effectiveness of the proposed method. Throughout the experiments, the polynomial approximations are represented using the Zernike polynomial basis. \\

\noindent 
We consider the following functions
\begin{equation*}
    f_1(x,y)=e^{-(x^2 + y^2)}, \quad f_2(x,y)=\sin (xy), \quad f_3(x,y)=e^{-xy}, \quad 
\end{equation*}
\begin{equation*}
f_4(x,y)=\frac{1}{x^2+y^2+1}, \quad 
    f_5(x,y)=\frac{1}{4x^2+4y^2+1}, \quad f_6(x,y)=\log (x^2+y^2+1).
\end{equation*}
\subsection{Function approximations}
For the first tests, we use the set of sampling nodes $\tilde{S}_{n,n}^1$, defined in~\eqref{pointsetadm}, and fix
\begin{equation*}
    m=\left\lfloor \frac{n}{4} \right \rfloor, \quad \tilde{r}=m+\left\lfloor\sqrt{m}\right\rfloor.
\end{equation*}
For each test function $f_i$, $i=1,\dots,6$, we compare the behaviour of the maximum approximation errors 
\begin{equation*}
e_{m}^{\mathrm{opt}}\left(f_i\right), \quad  e_{m}^{\prime}\left(f_i\right), \quad \hat{e}_{\tilde{r},n}\left(f_i\right)   
\end{equation*}
obtained by approximating each $f_i$ using
\begin{itemize}
    \item[-]  the polynomial interpolation on the set of optimal nodes $\Sigma_m^{\mathrm{opt}}$,
    \item[-] the polynomial interpolation on the set of mock-optimal nodes $\Sigma_m^{\prime}$,
    \item[-] the polynomial approximation $\hat{\Pi}_{\tilde{r},n}\left[f_i\right]$, defined in~\eqref{operatorPhat},
\end{itemize}
respectively.
This comparison is performed by increasing $n$ from $n=10$ to $n=100$. The results are shown in Fig.~\ref{ex1} and Fig.~\ref{ex2}. From these plots, it is clear that the error trends are similar. More precisely, as $n$ increases, the errors decrease until they reach machine precision, after which they remain constant. In particular, the interpolation error computed on the optimal nodes exhibits the same behavior as that obtained using the mock-optimal nodes. Moreover, the addition of simultaneous regression on the remaining nodes further improves the approximation accuracy.\\

\begin{figure}
  \centering
\includegraphics[width=0.32\textwidth]{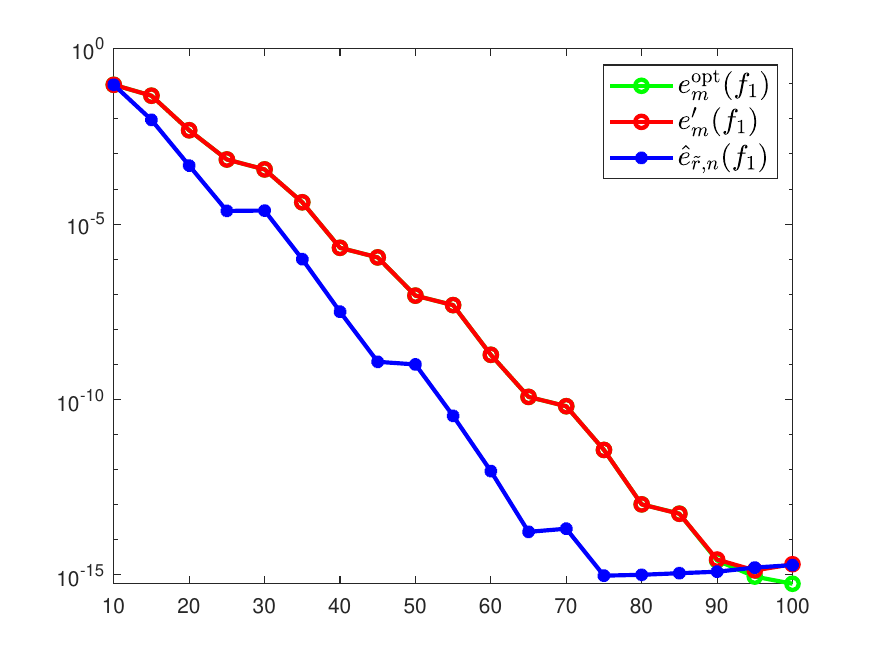} 
\includegraphics[width=0.32\textwidth]{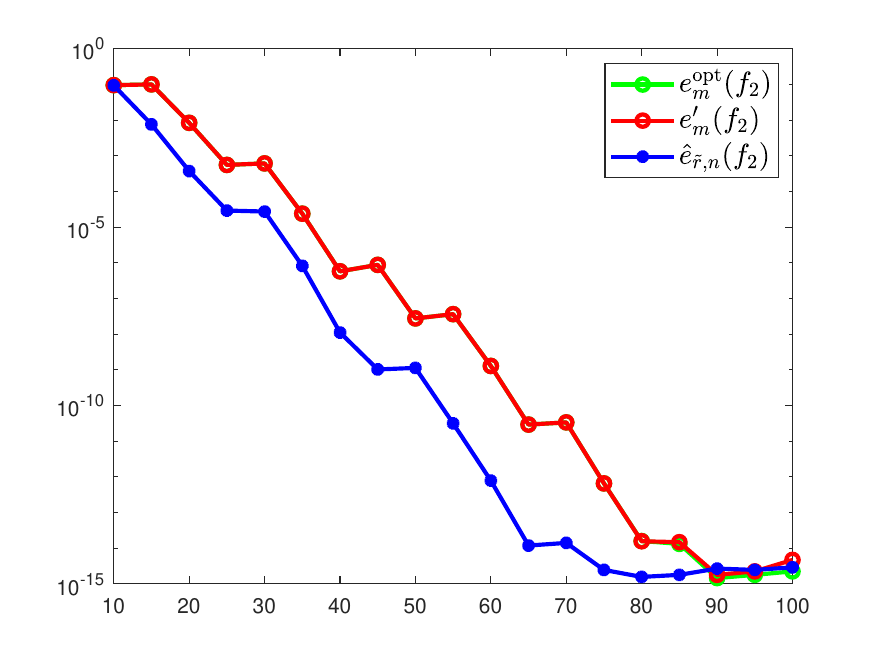} 
\includegraphics[width=0.32\textwidth]{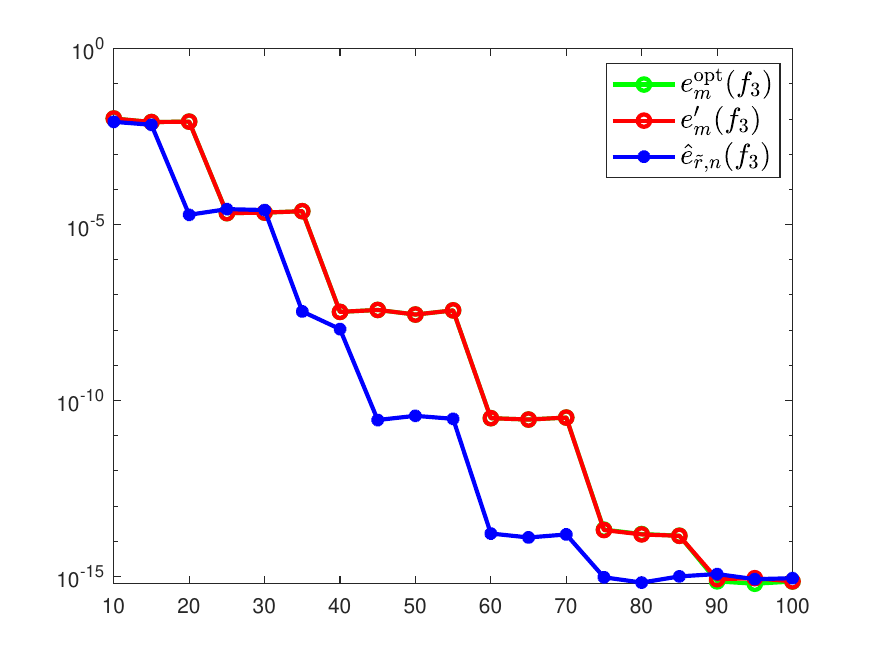}
 \caption{Comparison between the maximum approximation errors for the functions $f_1$ (left), $f_2$ (center), $f_3$ (right) produced by the polynomial interpolation on the set of optimal nodes $\Sigma_m^{\mathrm{opt}}$ ($e_{m}^{\mathrm{opt}}\left(f_i\right)$), the polynomial interpolation on the set of mock-optimal nodes $\Sigma_m^{\prime}$ ($e_{m}^{\prime}\left(f_i\right)$), and by $\hat{\Pi}_{\tilde{r},n}\left[f_i\right]$ ($\hat{e}_{\tilde{r},n}\left(f_i\right)$) as $n$ increases from $n=10$ to $n=100$.}
 \label{ex1}
\end{figure}

\begin{figure}
  \centering
\includegraphics[width=0.32\textwidth]{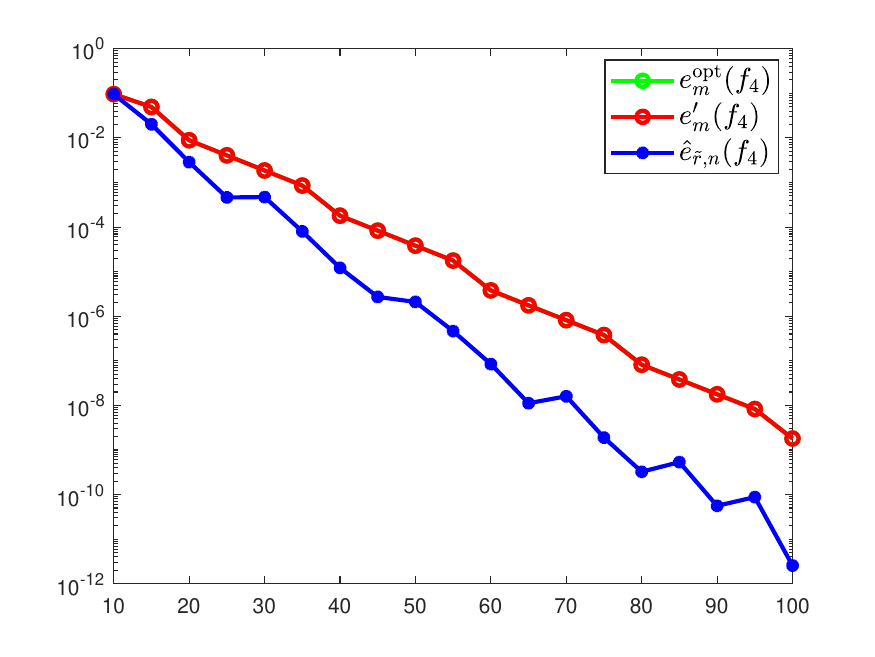}
\includegraphics[width=0.32\textwidth]{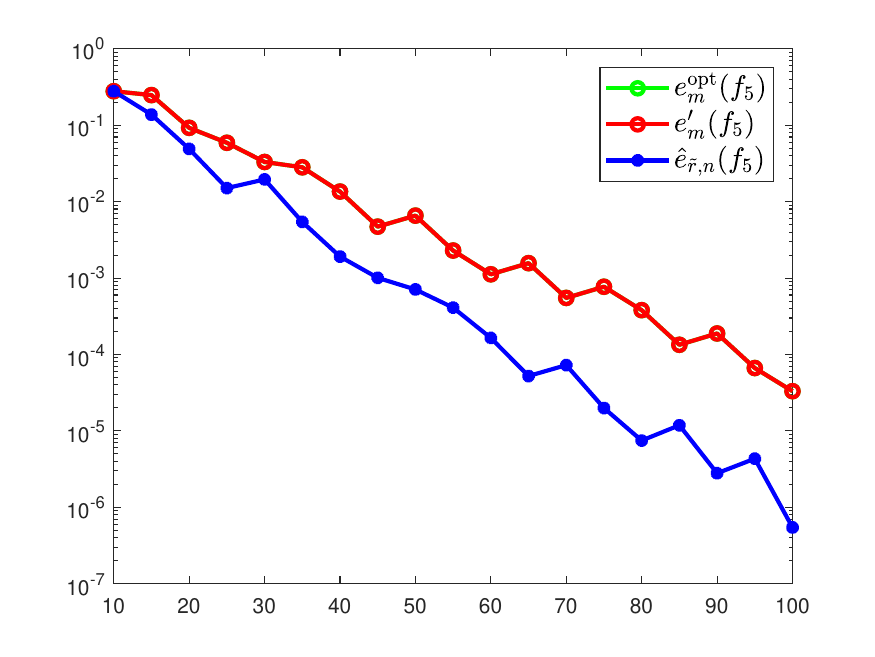}
\includegraphics[width=0.32\textwidth]{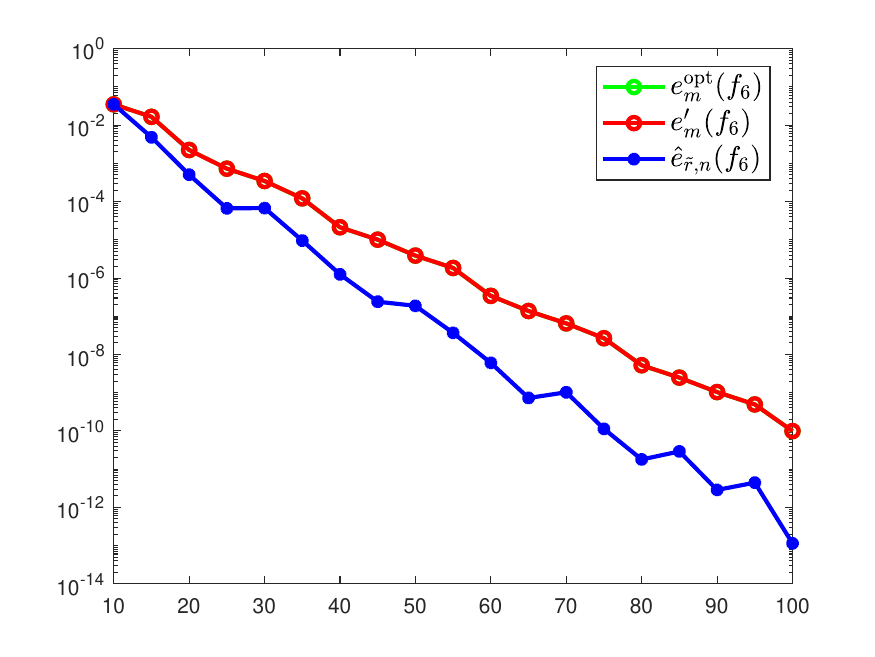}
 \caption{Comparison between the maximum approximation errors for the functions $f_4$ (left), $f_5$ (center), $f_6$ (right) produced by the polynomial interpolation on the set of optimal nodes $\Sigma_m^{\mathrm{opt}}$ ($e_{m}^{\mathrm{opt}}\left(f_i\right)$), the polynomial interpolation on the set of mock-optimal nodes $\Sigma_m^{\prime}$ ($e_{m}^{\prime}\left(f_i\right)$), and by $\hat{\Pi}_{\tilde{r},n}\left[f_i\right]$ ($\hat{e}_{\tilde{r},n}\left(f_i\right)$) as $n$ increases from $n=10$ to $n=100$.}
 \label{ex2}
\end{figure}

\noindent
Since the accuracy of the approximation produced by $\hat{\Pi}_{\tilde{r},n}$ depends not only on the set of nodes but also on the values of $m$ and $\tilde{r}$, in the second set of numerical experiments we compute the errors $\hat{e}_{\tilde{r},n}\left(f_i\right)$, $i = 1,\dots,6$, while varying the values of $m$ and $\tilde{r}$ with respect to the two sets of points 
\begin{equation*} \tilde{S}_{100,100}^1 \quad \text{and} \quad S^{\mathrm{spiral}}_{10000}, \end{equation*} 
defined in~\eqref{pointsetadm} and~\eqref{spiralnodes}, respectively. The results are shown in Tables~\ref{tab:gaussmax} and~\ref{tab:gaussmax1}. From these tables, we observe that the values produced with respect to the two different sets are comparable, and that increasing the values of $m$ and $\tilde{r}$ can significantly improve the quality of the approximation.
\begin{table}[ht!]
    \centering
    \begin{tabular}{|c|c|c|c|c|c|c|}\hline
        & $(m,\tilde{r})=(5,10)$ & $(m,\tilde{r})=(10,15)$ & $(m,\tilde{r})=(15,20)$& $(m,\tilde{r})=(20,25)$ & $(m,\tilde{r})=(25,30)$ \\ \hline
        $f_1$ &  1.6031e-06 	 & 1.3714e-09 	& 7.1942e-14 & 6.0840e-14 	 & 6.1118e-14   \\
        $f_2$ &  6.7861e-07 & 3.2387e-10 	 &  9.4369e-15 	 & 2.4425e-15 &  4.4409e-15  \\
        $f_3$ & 2.8639e-08 	 & 2.1369e-11 	 &  9.1593e-15 	 & 7.7716e-16 &   7.2164e-16   \\
        $f_4$ &  1.4543e-04 	 & 3.0322e-06 	 &  1.2572e-08 	 &   3.1671e-10 & 2.5564e-12 \\
         $f_5$ &   1.1938e-02 	 &  1.1529e-03 	 & 5.3044e-05 	 & 7.3282e-06 &  5.4555e-07   \\
          $f_6$ & 1.7005e-05 	 & 2.6903e-07 	 & 8.0837e-10 	 & 1.7393e-11 & 1.1391e-13    \\\hline
    \end{tabular}
    \caption{Maximum approximation errors $\hat{e}_{\tilde{r},100}\left(f_i\right)$, $i=1,\dots,6$ computed on the set $\tilde{S}_{100,100}^1$ for different values of $m$ and $\tilde{r}$.}
    \label{tab:gaussmax}
\end{table}
\begin{table}[ht!]
    \centering
    \begin{tabular}{|c|c|c|c|c|c|c|}\hline
         & $(m,\tilde{r})=(5,10)$ & $(m,\tilde{r})=(10,15)$ & $(m,\tilde{r})=(15,20)$& $(m,\tilde{r})=(20,25)$ & $(m,\tilde{r})=(25,30)$ \\ \hline
        $f_1$ &  1.4077e-06 	 & 1.4599e-09 	&  2.1649e-14 &  2.1094e-15 	 & 1.9429e-15   \\
        $f_2$ &   6.9109e-07 & 2.9659e-10 	 &  9.9920e-15 	 & 2.4425e-14 &   2.3093e-14  \\
        $f_3$ &  3.0483e-08 	 &  2.1060e-11 	 &   1.1990e-14 	 &  2.7478e-14 &   2.8866e-14   \\
        $f_4$ &  1.3779e-04 	 & 2.9242e-06 	 &  1.3267e-08	 &   4.5378e-10 & 3.1434e-12 \\
         $f_5$ &   1.2586e-02 	 &  1.0596e-03 	 & 6.2751e-05 	 &  7.5731e-06 &  8.2208e-07   \\
          $f_6$ &  1.5903e-05 	 &  2.6311e-07 	 &  8.6596e-10	 & 2.4996e-11 &    1.3810e-13    \\\hline
    \end{tabular}
    \caption{Maximum approximation errors $\hat{e}_{\tilde{r},100}\left(f_i\right)$, $i=1,\dots,6$ computed on the set ${S}^{\text{spiral}}_{10000}$ for different values of $m$ and $\tilde{r}$.}
    \label{tab:gaussmax1}
\end{table}
\noindent
However, increasing the value of $m$ does not necessarily lead to a better approximation. Specifically, by fixing $n = 100$, we analyze the behavior of the approximation errors $\hat{e}_{\tilde{r},n}\left(f_i\right)$, $i = 1, \dots, 6$, computed on the set $\tilde{S}_{100,100}^1$ as $m$ varies from $m=10$ to $m=70$, with $\tilde{r} = m + \left\lfloor \sqrt{m} \right\rfloor$. The results of this experiment are shown in Fig.~\ref{ex1m} and Fig.~\ref{ex2m}. These plots clearly show that the error decreases only up to a certain value of $m$, beyond which it increases significantly. This behavior is due to the condition number of the corresponding KKT matrix, as illustrated in Fig.~\ref{cond}.

\begin{figure}
  \centering
\includegraphics[width=0.32\textwidth]{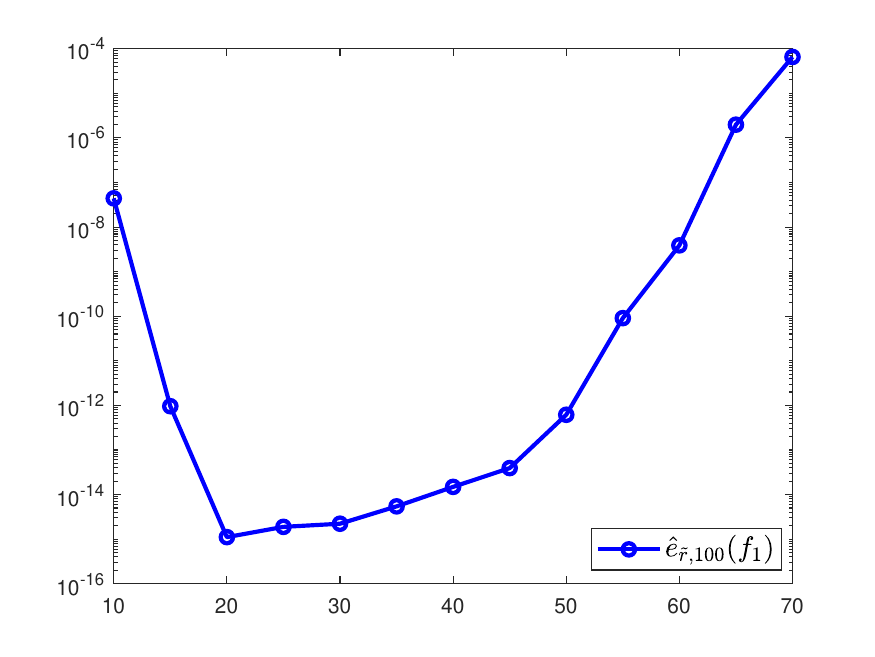} 
\includegraphics[width=0.32\textwidth]{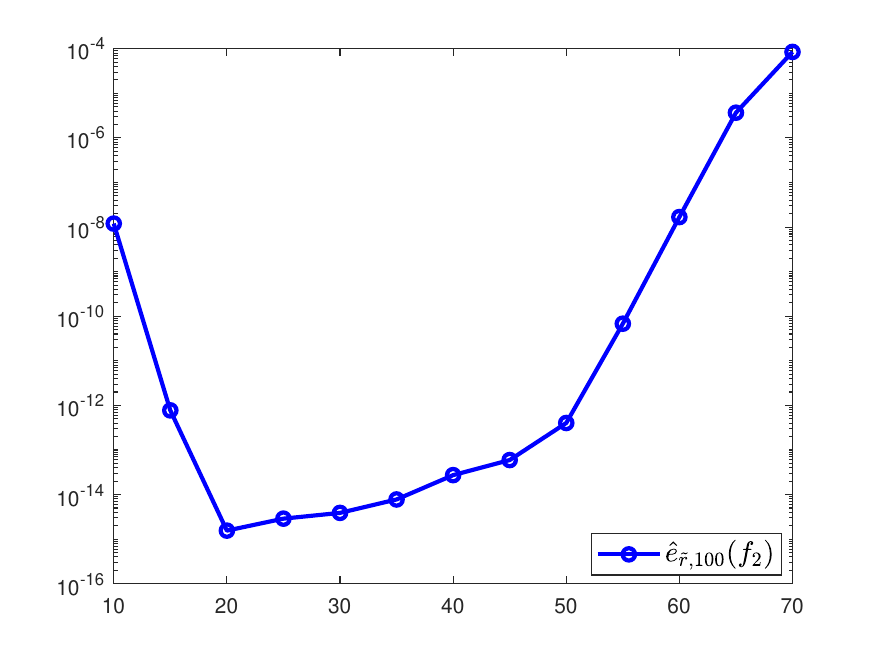} 
\includegraphics[width=0.32\textwidth]{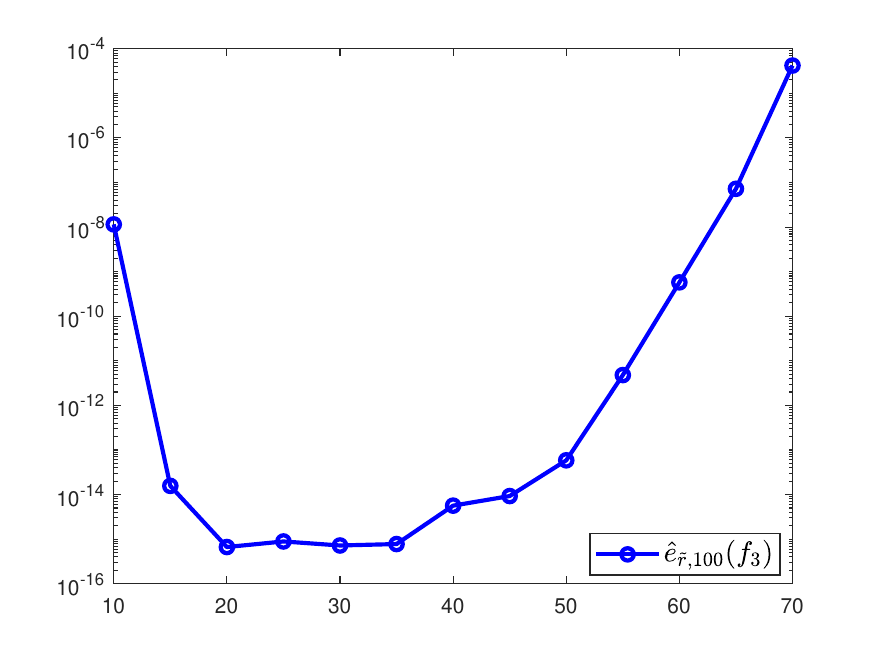}
 \caption{Trend of the maximum approximation errors for the functions $f_1$ (left), $f_2$ (center), $f_3$ (right) produced by  $\hat{\Pi}_{\tilde{r},100}\left[f_i\right]$ ($\hat{e}_{\tilde{r},100}\left(f_i\right)$) computed on the set $\tilde{S}_{100,100}^1$, with $m$ varying from $m=10$ to $m=70$, and $\tilde{r} = m + \left\lfloor \sqrt{m} \right\rfloor$.}
 \label{ex1m}
\end{figure}

\begin{figure}
  \centering
\includegraphics[width=0.32\textwidth]{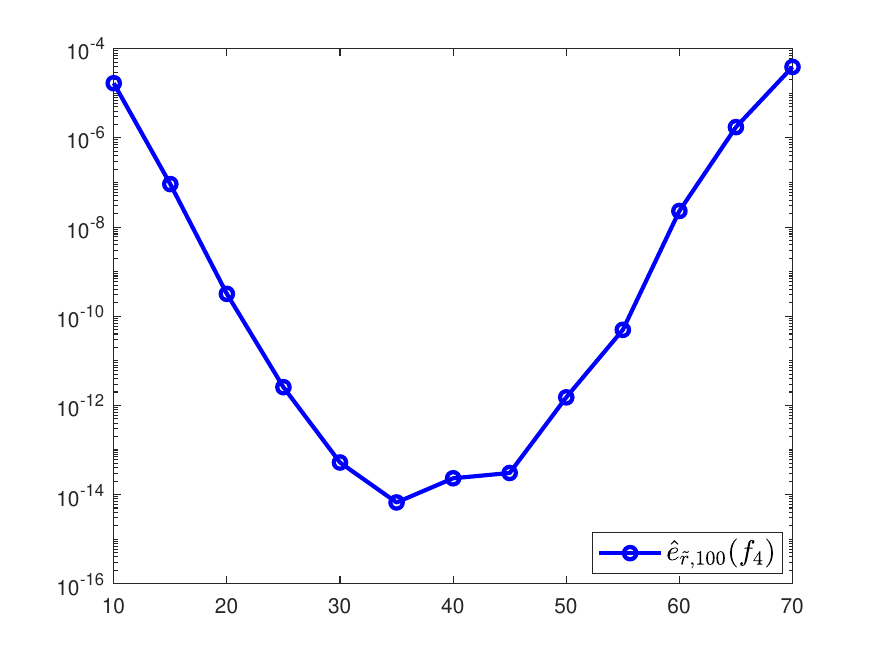}
\includegraphics[width=0.32\textwidth]{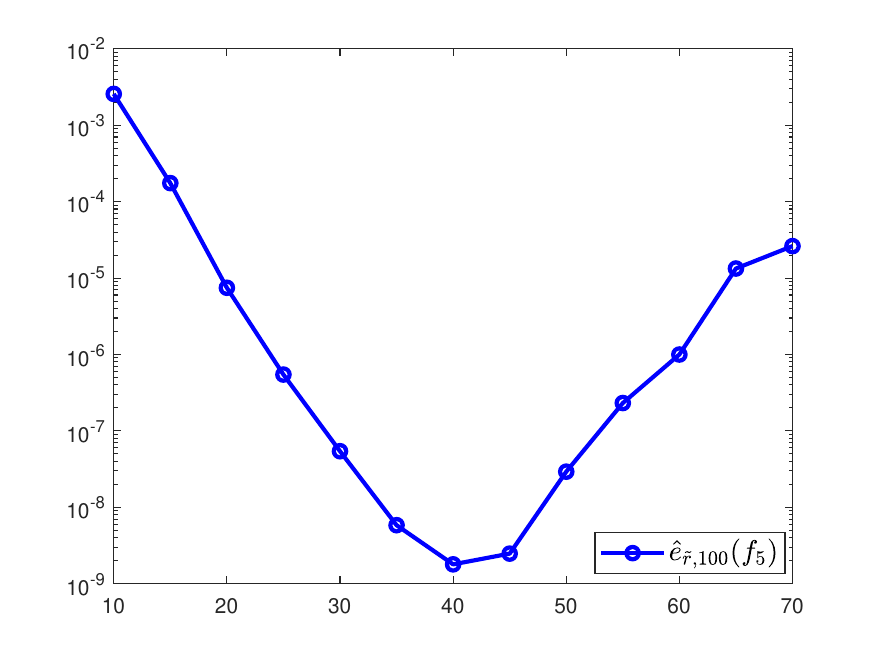}
\includegraphics[width=0.32\textwidth]{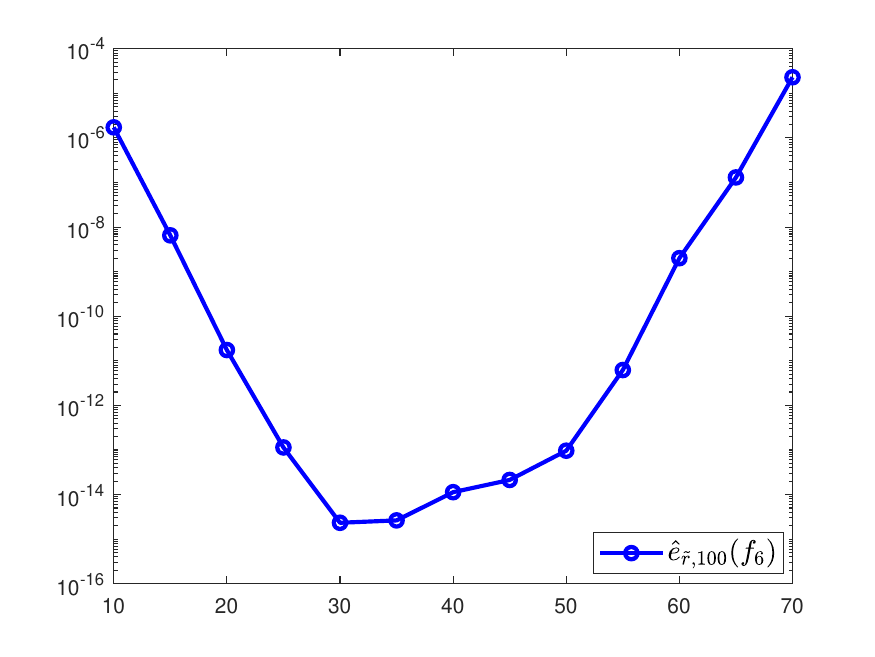}
 \caption{Trend of the maximum approximation errors for the functions $f_4$ (left), $f_5$ (center), $f_6$ (right) produced by  $\hat{\Pi}_{\tilde{r},100}\left[f_i\right]$ ($\hat{e}_{\tilde{r},100}\left(f_i\right)$) with $m$ varying from $m=10$ to $m=70$, and $\tilde{r} = m + \left\lfloor \sqrt{m} \right\rfloor$.}
 \label{ex2m}
\end{figure}

\begin{figure}
  \centering
\includegraphics[width=0.49\textwidth]{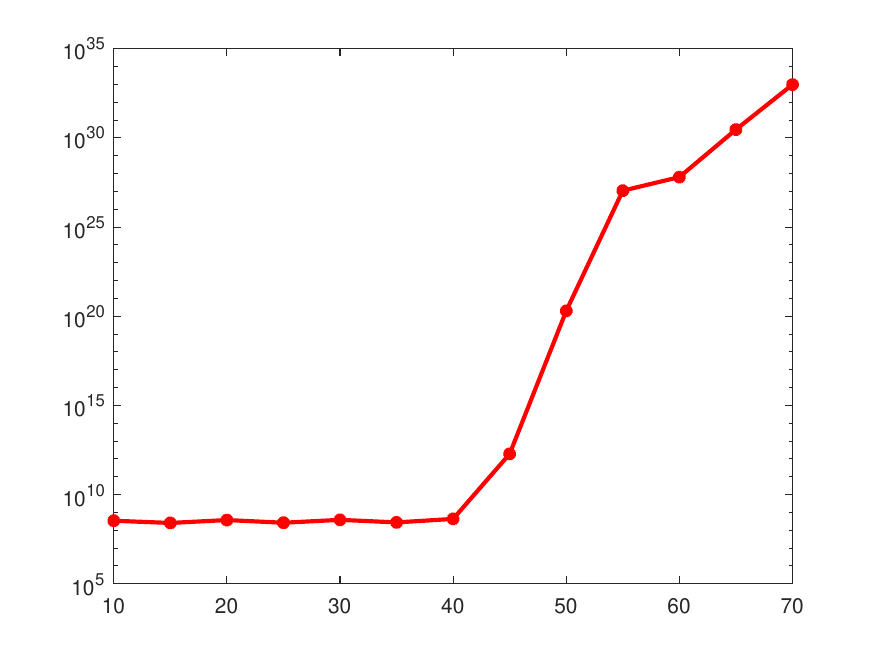}
 \caption{Trend of the condition number of the KKT matrix computed for $n=100$, with $m$ varying from $m=10$ to $m=70$, and $r = m + \left\lfloor \sqrt{m} \right\rfloor$.}
 \label{cond}
\end{figure}

\noindent
Finally, to assess the quality of the approximation, we use the operator $\hat{\Pi}_{\tilde{r},n}$ to reconstruct the functions $f_i$, $i=1,\dots,6$. In particular, we set $n=80$, $m=20$, $\tilde{r}=24$ and consider the set $\tilde{S}_{n,n}^1$. The results are shown in Fig.~\ref{ex1plot}.

\begin{figure}
  \centering
\includegraphics[width=0.32\textwidth]{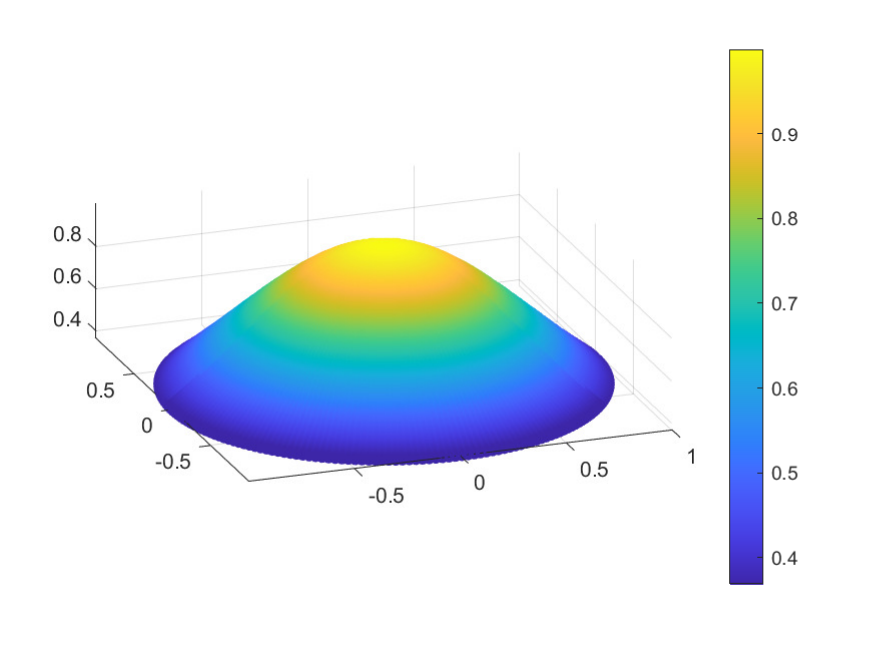} 
\includegraphics[width=0.32\textwidth]{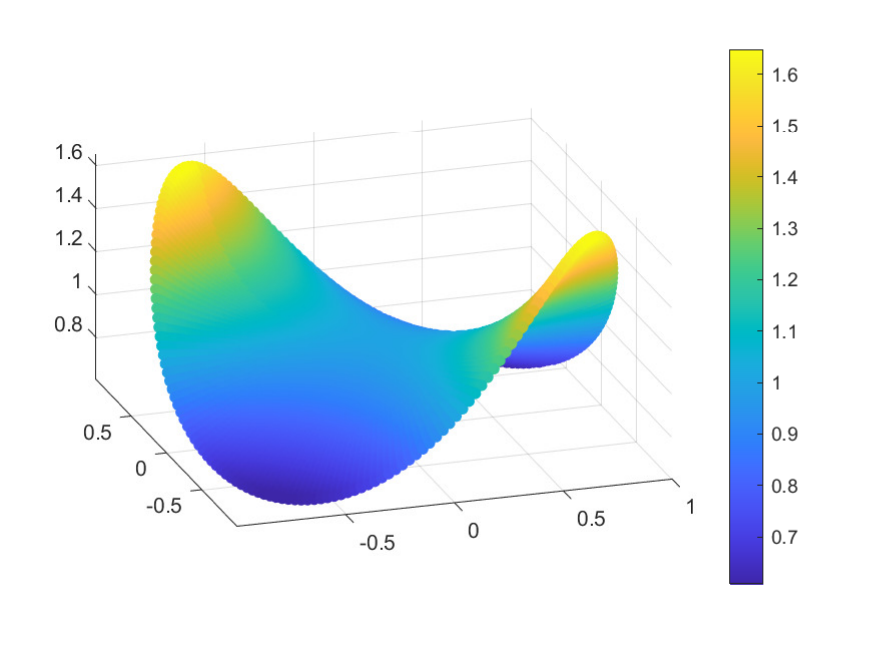} 
\includegraphics[width=0.32\textwidth]{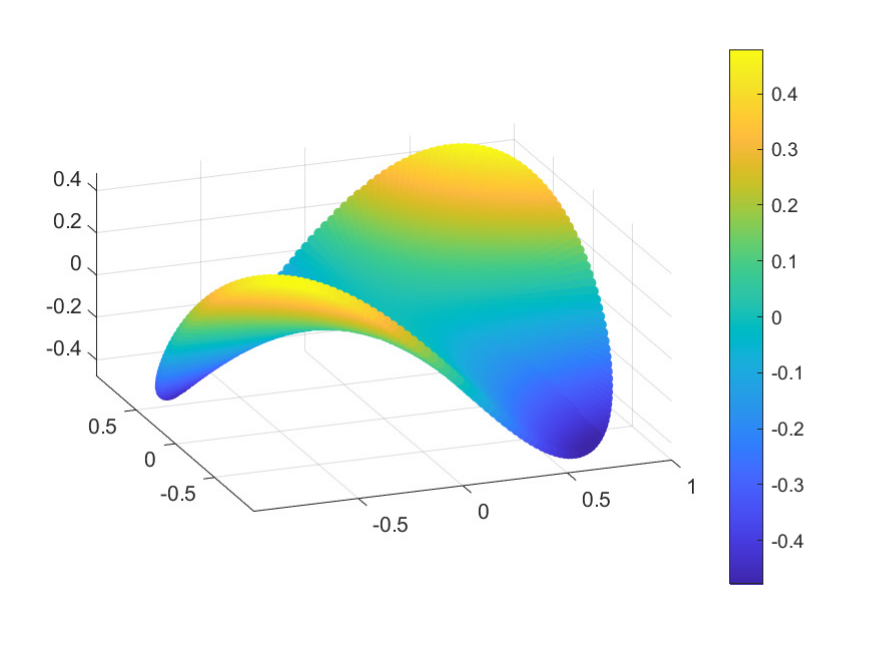}
\includegraphics[width=0.32\textwidth]{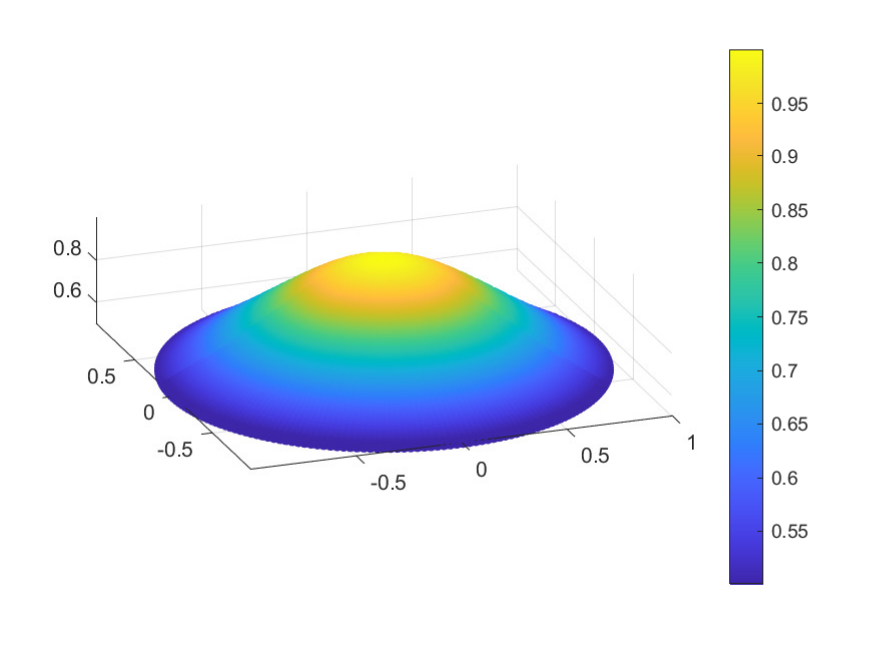} 
\includegraphics[width=0.32\textwidth]{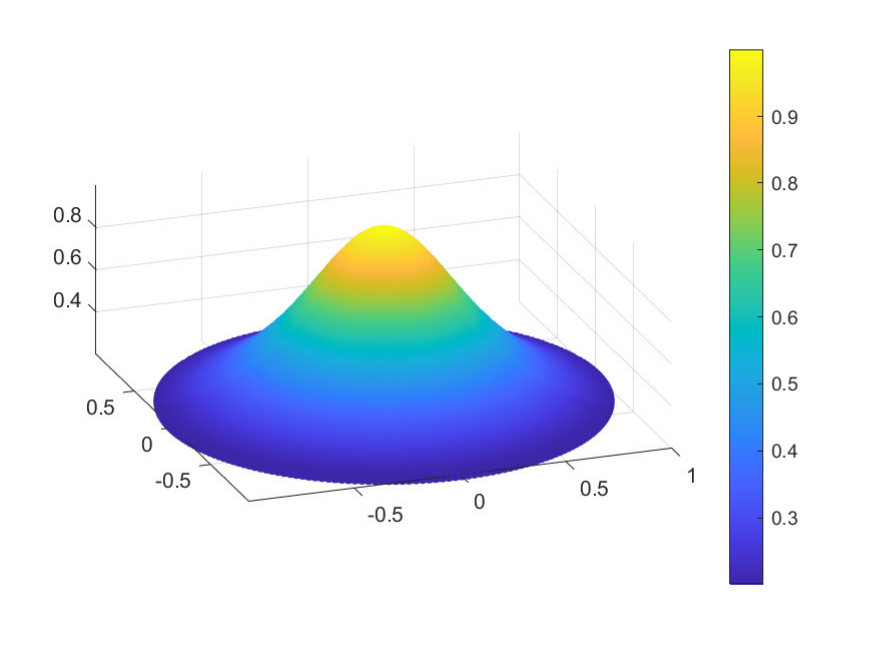} 
\includegraphics[width=0.32\textwidth]{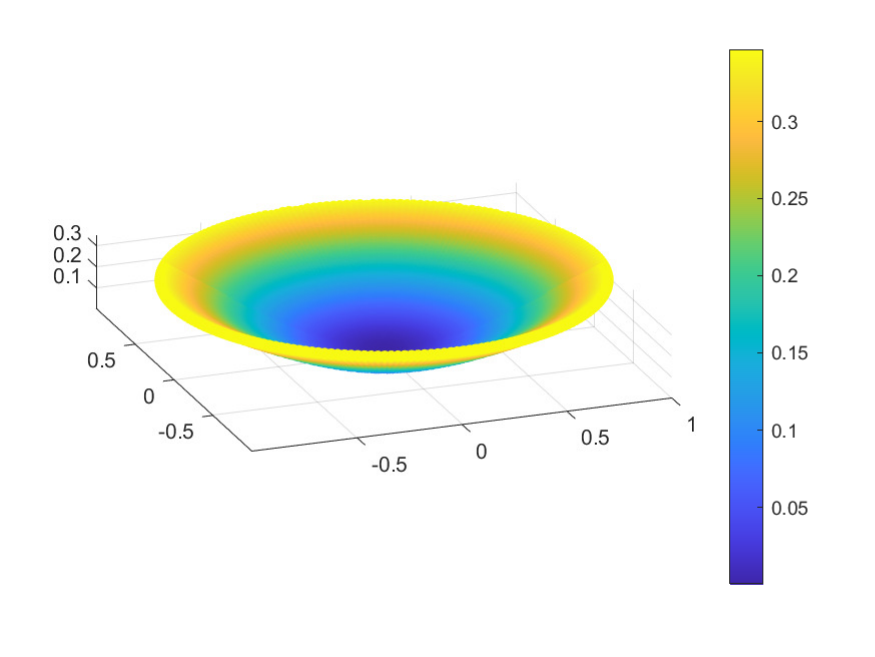}
 \caption{Reconstruction of the functions $f_1$ (left), $f_2$ (center), $f_3$ (right) (top) and  $f_4$ (left), $f_5$ (center) and $f_6$ (right) (bottom) using the operator $\hat{\Pi}_{\tilde{r},n}$ on the set of nodes $\tilde{S}_{80,80}^1$.}
 \label{ex1plot}
\end{figure}

\subsection{Cubature formula}
In this section, we test the accuracy of the cubature formula~\eqref{quadFor} for approximating the integrals
\begin{equation}\label{ints}
   I\left[f_i\right]=\int_{D_1}f_i(x,y)dxdy, \quad i=1,\dots,6,
\end{equation}
where $D_1$ denotes the unit disk centered at $\boldsymbol{0}=(0,0)$.
To this aim, we set 
\begin{equation*}
    m=\left\lfloor \frac{n}{4} \right\rfloor, \quad r=m+\left\lfloor m\right\rfloor, \quad P=10,
\end{equation*}
and compute the errors
\begin{equation*}
    e_{n,P}\left[f_i\right]=\left\lvert I\left[f_i\right]-{\hat{Q}_{\tilde{r},n,P}\left[f_i\right]}\right\rvert, \quad i=1,\dots,6,
\end{equation*}
obtained from the operator $\hat{\Pi}_{\tilde{r},n}$ relative to the set $\tilde{S}^1_{n,n}$ with $n=20k$, $k=1,\dots,5$.  The software to compute the cubature nodes on the disk is available at the website \url{https://www.math.unipd.it/~alvise/sets.html}.

\begin{table}[ht!]
    \centering
    \begin{tabular}{|c|c|c|c|c|c|}\hline
      $e_{n,10}$  & $n=20$ & $n=40$ & $n=60$ & $n=80$ & $n=100$\\ \hline
        $f_1$ &    1.3546e-04 	 &  2.2690e-09	 &  4.6141e-13 	 &   8.4377e-15 & 3.9968e-14 \\
        $f_2$ &  9.2933e-06 &  5.2194e-10 & 2.9310e-14 & 4.8850e-15 &   2.2204e-15  \\
        $f_3$ &   2.6073e-17 &    8.8915e-18 &   2.9867e-17 	 &     1.3612e-17 &  3.0595e-17  \\
        $f_4$ &   1.0660e-03  &  5.9728e-07   &  5.0668e-08	 	 &  5.5038e-11   & 2.0104e-12\\
        $f_5$ &   2.1286e-02 &  5.2240e-04 & 1.1999e-04 	 &  6.1720e-07 & 5.6009e-07  \\
        $f_6$ &  1.7709e-04  &  3.6233e-08   &  3.5316e-09	 	 & 3.2389e-12   & 8.8152e-14\\\hline
    \end{tabular}
    \caption{Approximation errors produced by approximating the integrals~\eqref{ints} through the cubature formula $\hat{Q}_{\tilde{r},n,P}$ on $P=10$ Gaussian cubature nodes.}
    \label{tab:cub}
\end{table}
\noindent
From Table~\ref{tab:cub}, we observe that the error decreases as the cardinality of the set $\tilde{S}_{n,n}^1$ increases. This trend aligns with the behavior of the approximation errors produced by the operator $\hat{\Pi}_{\tilde{r},n}$.

\section{Conclusions and Future Works}
In this paper, we have introduced a new interpolation-regression approximation operator for reconstructing functions defined on disk domains. By extending the constrained mock-Chebyshev least squares framework to circular geometries, we have successfully combined an optimal interpolation strategy (via the Bos array and Zernike polynomials) with a simultaneous regression step that exploits all available sampling data. Our analysis provided rigorous error bounds for the norm of the interpolation-regression operator, and numerical experiments confirmed the effectiveness of the method for both function reconstruction and the derivation of high-accuracy cubature formulas on the disk. The proposed methodology not only circumvents the challenges associated with standard interpolation (such as the Runge phenomenon) but also leverages the underlying structure of the optimal nodes to achieve stability and high-accuracy approximations.  These results demonstrate the potential of the approach in various applications, including optical engineering and numerical integration over circular domains.

Future research directions include:
\begin{itemize}
    \item \textbf{Extension to Other Domains:} Investigating the applicability of the interpolation-regression framework to more general domains (e.g., ellipsoidal or higher-dimensional spherical domains) and to function spaces defined on manifolds.
    \item \textbf{Robustness and Efficiency:} Enhancing the numerical stability and computational efficiency of the method, particularly for large-scale problems or when dealing with noisy data.
    \item \textbf{Generalized Weight Functions:} Extending the analysis to cover a broader class of weight functions and studying their impact on the convergence properties of the operator.
    \item \textbf{Applications:} Applying the developed techniques to real-world problems in areas such as image processing, optical aberration correction, and other fields where disk-shaped domains naturally arise.
\end{itemize}

\section{Declarations}
\begin{itemize}
\item Corresponding author: Federico Nudo, email address: federico.nudo@unical.it
    \item Funding: This research was supported by the GNCS-INdAM 2025 project \lq\lq Polinomi, Splines e Funzioni Kernel: dall’Approssimazione Numerica al Software Open-Source\rq\rq. The work of F. Marcellán has also been supported by the research project PID2021-122154NB-I00 \emph{Ortogonalidad y Aproximación con Aplicaciones en Machine Learning y Teoría de la Probabilidad}, funded by MICIU/AEI/10.13039/501100011033 and by \lq\lq ERDF A Way of Making Europe\rq\rq.
    \item Authors' Contributions: All authors contributed equally to this article, and therefore the order of authorship is alphabetical.
    \item Conflicts of Interest: The authors declare no conflicts of interest.
    \item Ethics Approval: Not applicable.
    \item Data Availability: No data were used in this study.
\end{itemize}

\section*{Acknowledgments}
 This research has been achieved as part of RITA \textquotedblleft Research
 ITalian network on Approximation'' and as part of the UMI group \enquote{Teoria dell'Approssimazione
 e Applicazioni}. The research was supported by GNCS-INdAM 2025 project \lq\lq Polinomi, Splines e Funzioni Kernel: dall’Approssimazione Numerica al Software Open-Source\rq\rq.  The work of F. Marcell\'an has been supported by the research project PID2021- 122154NB-I00] \emph{Ortogonalidad y Aproximación con Aplicaciones en Machine Learning y Teoría de la Probabilidad} funded  by MICIU/AEI/10.13039/501100011033 and by \lq\lq ERDF A Way of making Europe\rq\rq.

\bibliographystyle{spmpsci}
\bibliography{bibliography}

 \end{document}